\pdfoutput=1

\documentclass[11pt]{article} 

\usepackage[utf8]{inputenc} 

\usepackage{geometry} 
\usepackage{graphicx} 

\usepackage{booktabs} 
\usepackage{array} 
\usepackage{paralist} 
\usepackage{verbatim} 
\usepackage{subfig} 
\usepackage{amssymb}
\usepackage{amsmath}
\usepackage{filecontents}
\usepackage{amsthm}
\usepackage{color} 
\usepackage{enumerate}

\usepackage{fancyhdr} 
\usepackage{sectsty}
\allsectionsfont{\sffamily\mdseries\upshape} 

\newtheorem{theorem}{Theorem}[section]
\newtheorem{lemma}{Lemma}[section]
\newtheorem{prop}{Proposition}[section]

\newtheorem*{propA}{Proposition A}
\newtheorem*{propB}{Proposition B}
\newtheorem*{propA2}{Proposition A'}
\newtheorem*{propB2}{Proposition B'}
\newtheorem*{prop1}{Property 1}
\newtheorem*{prop2}{Property 2}
\newtheorem*{prop3}{Property 3}
\newtheorem*{prop12}{Property 1'}
\newtheorem*{prop22}{Property 2'}
\newtheorem*{prop32}{Property 3'}

\theoremstyle{definition}
\newtheorem{definition}{Definition}[section]

\theoremstyle{remark}
\newtheorem*{remark}{Remark}

\renewcommand{\Re}{\operatorname{Re}}
\renewcommand{\Im}{\operatorname{Im}}

\date{} 

\title{A polynomial automorphism with a wandering Fatou component}
\author{David Hahn, Han Peters}

\begin{document}

\renewcommand{\baselinestretch}{1.07}

\maketitle

\begin{abstract}
We construct polynomial automorphisms with wandering Fatou components. The four-dimensional automorphisms $H$ lie in a one-parameter family, depending on the parameter $\delta \in \mathbb C \setminus \{0\}$, and as $\delta \rightarrow 0$ the automorphisms degenerate to the two-dimensional polynomial map $P$ constructed in \cite{ABDPR}. Our main result states that if $P$ has a wandering domain, then $H$ does too for $\delta$ sufficiently small.
\end{abstract}

\section{Introduction}

Sullivan's No Wandering Domain Theorem \cite{Sullivan} asserts that polynomials and rational functions in $\mathbb{C}$ do not have wandering Fatou components.
Quite recently it was shown in \cite{ABDPR} that in higher dimensions there do exist polynomial maps with wandering Fatou components. The two-dimensional maps constructed in \cite{ABDPR} have the simple form
$$
P: (z,w) \mapsto \left(f(z) + \frac{\pi^2}{4} \cdot w, g(w)\right),
$$
where $f,g$ are polynomials in one variable. Important in the construction is that both $f$ and $g$ have a parabolic fixed point at the origin.

Here we will consider a problem suggested to us by Romain Dujardin: \emph{Use the techniques introduced in \cite{ABDPR} to construct polynomial automorphisms with wandering Fatou components.}

An immediate observation is that this is impossible in dimension $2$. Polynomial automorphisms have constant Jacobian determinant, and the construction with the two parabolic fixed points requires this Jacobian determinant to have norm $1$. But a volume preserving map cannot have wandering Fatou components, at least not with bounded orbits.

Thus in order to search for polynomial automorphisms with wandering Fatou components, using similar techniques, we are forced to consider higher dimensional maps. The maps that we will consider are small four-dimensional perturbations of the polynomial $P$. An initial idea is to consider invertible maps of the form
$$
H(z, x) = (P(z) - \delta \cdot x, z),
$$
for $z,x \in \mathbb C^2$ and $\delta \in \mathbb C \setminus \{0\}$. These maps are four-dimensional analogues of H\'enon maps, degenerating to the polynomial map $P$ as $\delta \rightarrow 0$. However, for these maps the fixed point $(0,0)$ is hyperbolic, i.e. none of the eigenvalues of $DH(0,0)$ have norm $1$, making it impossible to use techniques from \cite{ABDPR}.

Instead, we will consider a similar class of invertible maps $H:\mathbb{C}^4\rightarrow\mathbb{C}^4$ of the form
\begin{equation}\label{eq:H}
H: \left(\left( \begin{array}{c} z \\ x \end{array}\right), \left( \begin{array}{c} w \\ y \end{array}\right)\right) \mapsto \left(F\left( \begin{array}{c} z \\ x \end{array}\right) + \left( \begin{array}{c} \frac{\pi^2}{4}w \\ 0 \end{array}\right), G\left( \begin{array}{c} w \\ y \end{array}\right) \right)
\end{equation}
where
\begin{align}\label{FGdef}
F\left( \begin{array}{c} z \\ x \end{array}\right) = \left( \begin{array}{c} z + q_1(z+ \delta x) \\ \delta x - q_1(z+ \delta x) \end{array}\right) \; \;
\mathrm{and} \; \;
G\left( \begin{array}{c} w \\ y \end{array}\right) = \left( \begin{array}{c} w + q_2(w+ \delta y) \\ \delta y - q_2(w+ \delta y)\end{array}\right).
\end{align}
Here the polynomials $q_1$ and $q_2$ are chosen such that
$$
f(z) = z + q_1(z) \; \; \mathrm{and} \; \; g(w) = w + q_2(w).
$$
Note that $F$ and $G$ are invertible polynomial maps, both conjugate to H\'enon maps, and that they degenerate to the functions $f$ and $g$ as $\delta\rightarrow 0$.

Both $F$ and $G$ have semi-parabolic fixed points at the origin. The existence of Fatou coordinates for such maps has been proved in \cite{Ueda1, Ueda2}. In \cite{BSU} it was shown that Lavaurs Theorem \cite{Lavaurs}, the main idea behind the proof in \cite{ABDPR}, also holds in the two-dimensional semi-parabolic setting, suggesting that it might be possible to prove the existence of wandering domains for the four-dimensional map $H$. Indeed, we will prove the following:

\begin{theorem}
Let $f,g$ be as in \cite{ABDPR}. For $\delta$ small enough the map $H$ defined in \eqref{eq:H} has wandering Fatou components.
\end{theorem}

As in \cite{ABDPR}, we do not directly apply Lavaurs Theorem, or its generalization proved in \cite{BSU}, but rather prove convergence to the Lavaurs map of $F$ for compositions of a sequence of perturbations of $F$. Our proof closely follows the proof in \cite{ABDPR}, comparing iterates of $H$ to translations in suitable Fatou coordinates. Often we will be able to directly use the estimates obtained in \cite{ABDPR} for the one-dimensional setting, without the need to redo the computations. In those cases we merely need to estimate the difference between the one-dimensional and the two-dimensional setting.

We note that for \emph{holomorphic} automorphisms there have been earlier constructions of wandering domains. In \cite{FS} holomorphic automorphisms of $\mathbb C^2$ with wandering domains were constructed, and in \cite{ABFP} this was done for transcendental H\'enon maps, a more restricted class of maps. In each of these cases the wandering domains had unbounded orbits, and the proofs relied on Runge approximation to control the orbits near infinity.

\medskip

It remains unknown whether there exist polynomial automorphisms of $\mathbb C^2$ or $\mathbb C^3$ with wandering Fatou components. In two complex variables this question is considered particularly interesting, but possibly quite difficult.

\section{Preliminaries and outline of the paper}

Let us recall the main result from \cite{ABDPR}.

\begin{theorem}\label{MainABDPR}
Let $f:\mathbb{C}\rightarrow\mathbb{C}, g:\mathbb{C}\rightarrow\mathbb{C}$ be polynomials such that
$$f(z)=z+z^2+az^3, \qquad g(w)=w-w^2+\mathcal{O}(w^3),$$
where $a\in D(1-r,r)$ for $r>0$ sufficiently small.
Then the map
$$P:\mathbb{C}^2\rightarrow\mathbb{C}^2, \qquad P(z,w):=(f(z)+\frac{\pi^2}{4}w, g(w))$$
admits a wandering Fatou component.
\end{theorem}

The proof of this result relied in an essential way on the notion of the \emph{Lavaurs map}, which we will now introduce.
For the function $f$ and the two-dimensional map $F$ as in \eqref{FGdef}, denote the parabolic basins of their fixed points at the origin by $\mathcal{B}_f$ and $\mathcal{B}_F$, respectively. We use the same notation for the basins $\mathcal{B}_g$ and $\mathcal{B}_G$ of $g$ and $G$.
It is known that there exist {\it Fatou coordinates} for such maps (the two-dimensional coordinates were introduced by Ueda in \cite{Ueda1, Ueda2}).
The attracting Fatou coordinate $\phi_f:\mathcal{B}_f\rightarrow\mathbb{C}$ of $f$ satisfies the functional equation
$$
\phi_f\circ f = T_1 \circ \phi_f,
$$
where we define $T_1$ to be the translation $T_1(Z)=Z+1$.
The same functional equation holds for the attracting Fatou coordinate $\Phi_F:\mathcal{B}_F\rightarrow \mathbb{C}$ of $F$.
The repelling Fatou coordinate $\psi_f:\mathbb{C}\rightarrow\mathbb{C}$ satisfies the equation
$$
f\circ \psi_f = \psi_f \circ T_1
$$
as does the two dimensional repelling Fatou coordinate $\Psi_F:\mathbb{C}\rightarrow\mathbb{C}^2$.
The (phase $0$) Lavaurs maps $\mathcal{L}_f:\mathcal{B}_f\rightarrow\mathbb{C}$ and $\mathcal{L}_F:\mathcal{B}_F\rightarrow\mathbb{C}^2$ are defined as the compositions
$$
\mathcal{L}_f:=\psi_f\circ\phi_f \qquad \rm{and}\qquad  \mathcal{L}_F:=\Psi_F\circ\Phi_F.
$$

The proof of Theorem \ref{MainABDPR} in \cite{ABDPR} followed quickly from the following two propositions:
\begin{propA}
For $|\delta| >0$ sufficiently small the sequence of maps
$$
\mathbb{C}^2\ni(z,w)\mapsto P^{2n+1}(z,g^{n^2}(w))\in\mathbb{C}^2
$$
converges locally uniformly in $\mathcal{B}_f\times\mathcal{B}_g$ to the map
$$
\mathcal{B}_f\times\mathcal{B}_g\ni (z,w)\mapsto(\mathcal{L}_f(z),0)\in\mathbb{C}\times\{0\}
$$
as $n\rightarrow\infty$.
\end{propA}
\begin{propB}
Let $f:\mathbb{C}\rightarrow\mathbb{C}$ be the polynomial
$$f(z):=z+z^2+az^3, \qquad a\in\mathbb{C}.$$
If $r>0$ is sufficiently small and $a\in D(1-r,r)$, then the Lavaurs map $\mathcal{L}_f:\mathcal{B}_f\rightarrow \mathbb{C}$ admits an attracting fixed point.
\end{propB}

Our goal is to prove the following two analogues:

\begin{propA2}
Let $H$, $F$ and $G$ be the maps as in \eqref{eq:H} and \eqref{FGdef}.
For $|\delta|>0$ small enough, the sequence of maps
$$
\mathbb{C}^4\ni(z,x,w,y)\mapsto H^{2n+1}(z,x,G^{n^2}(w,y))\in\mathbb{C}^4,
$$
converges locally uniformly in $\mathcal{B}_F\times\mathcal{B}_G$ to the map
$$
\mathcal{B}_F\times\mathcal{B}_G\ni (z,x,w,y)\mapsto(\mathcal{L}_F(z,x),0,0)\in\mathbb{C}^2\times\{0\}\times\{0\},
$$
as $n\rightarrow\infty$.
\end{propA2}

\begin{propB2}
Let $F$ be the map as in \eqref{FGdef}. Suppose the Lavaurs map $\mathcal{L}_f$ of $f$ has an attracting fixed point in $\mathcal{B}_f$. Then for $|\delta|>0$ small enough, the Lavaurs map $\mathcal{L}_F$ of $F$ has an attracting fixed point $(\hat{z}, \hat{x}) \in \mathcal{B}_F$.
\end{propB2}

Once these two propositions are proven, the existence of the wandering Fatou components follows exactly as in \cite{ABDPR}: There exists an open set $U \subset \mathcal{B}_F\times \mathcal{B}_G$ on which the sequence $H^n$ is bounded and the subsequence $H^{j^2}$ converges uniformly to the point $(\hat{z}, \hat{x}, 0,0)$. Since this point is not periodic for the map $H$, it follows that $U$ is contained in a wandering Fatou component.

\begin{remark}
The assumption that $\delta$ is small in Proposition A' is not necessary, with a little more effort the proposition can be proved for any $|\delta|<1$. However, it will be convenient to assume that $|\delta|<\delta_0<\frac{1}{4\pi^4}$. Moreover, the assumption that $\delta$ is sufficiently small is necessary for Proposition B', and we therefore see no reason for proving Proposition A' under the weaker assumption $|\delta| < 1$.

In addition, it will be clear from the proof that it is not necessary for $G$ to have the same Jacobian determinant $\delta$ as $F$, nor is it necessary for $G$ to have particularly small Jacobian determinant: $|\delta|<1$ is sufficient. Instead of introducing more notation, we have chosen to work with the same constant $\delta$.
\end{remark}

In section (3) we will first make a convenient change of coordinates, a slight modification of the coordinates changes introduced by Ueda in \cite{Ueda1, Ueda2}, and use these coordinates to introduce Fatou coordinates that vary holomorphically with $\delta$. As a consequence we prove Proposition B'. In section (4) we will introduce approximate Fatou coordinates, and use these in section (5) to prove Proposition A'. These two sections follow the presentation in \cite{ABDPR}, and our notation does as well, sometimes using capital letters in higher dimensions to distinguish from the one-dimensional setting.

\section{Local coordinates, degeneration of Fatou coordinates and the proof of Proposition B'}

This section introduces a local coordinate change and investigates the dependence of two-dimensional Fatou coordinates on the parameter $\delta$. This will be used for the proof of Proposition B'.

\subsection{Setting}

In \cite{Ueda1} and \cite{Ueda2} Ueda has investigated the local behaviour of holomorphic mappings $T$ that are defined locally near the origin $\mathcal{O}$ of $\mathbb{C}^2$ and map into $\mathbb{C}^2$ such that $\mathcal{O}$ is fixed.
Of particular interest is the semi-parabolic semi-attracting case, where the eigenvalues of the jacobian matrix $DT$ of $T$ are $1$ and $\delta$, with $|\delta|<1$. After a local coordinate change if necessary, we can assume that $DT$ is diagonalized and $T$ is of the form
\begin{align}\label{T}
T\begin{pmatrix}z\\x \end{pmatrix} =  \begin{pmatrix}z+\sum_{i+j\geq 2} a_{i,j} z^i x^j \\ \delta x + \sum_{i+j\geq 2} b_{i,j} z^i x^j \end{pmatrix}
\end{align}
with $a_{i,j}, b_{i,j} \in\mathbb{C}$.
Ueda shows in \cite{Ueda1} that for any integers $i, j\geq 1$ there exists a local coordinate system in which the map $T$ is of the form
\begin{align}\label{TildeT}
\widetilde{T}\begin{pmatrix}x\\y \end{pmatrix} = \begin{pmatrix}z+a_2 z^2 \dots + a_i z^i +a_{i+1}(x)z^{i+1}+\dots\\ \delta x + b_1xz + b_2 x z^2 +\dots +b_j x z^j + b_{j+1}(x) z^{j+1} +\dots \end{pmatrix},
\end{align}
where $a_2, \dots, a_i\in\mathbb{C}$ and $b_1, \dots, b_j\in\mathbb{C}$ are constant and $a_{i+1}, a_{i+2}, \dots$ and $b_{j+1}, b_{j+2}, \dots $ are functions depending on $x$.
Based on this local form, Ueda introduces two-dimensional Fatou coordinates.

Here, instead of the general semi-parabolic semi-attracting mapping $T$, we will consider the family of polynomials $F^\delta:\mathbb{C}^2\rightarrow\mathbb{C}^2$ of the form
\begin{align}\label{Fdelta}
F^\delta \left( \begin{array}{c} z \\ x \end{array}\right) = \left( \begin{array}{c} z + q(z+ \delta x) \\ \delta x - q(z+ \delta x)\end{array}\right),
\end{align}
as in \eqref{FGdef}, where $|\delta|<1$ and $q(z)$ is a polynomial of the form $q(z)=z^2+O(z^3)$. As remarked earlier, the maps $F^\delta$ are automorphisms for $0<|\delta|<1$, in fact they are conjugate to Hénon maps. As $\delta\rightarrow 0$ these maps degenerate to a one-dimensional polynomial, that is
$$
\pi_z\circ F^0(z,x) = z + q(z) = : f(z).
$$
Here and for the rest of this paper $\pi_z$ denotes the projection to the first coordinate $(z,x)\mapsto z$. For $0<|\delta|<1$, the maps $F^\delta$ are a special case of Ueda's maps $T$. The coordinate changes introduced by Ueda do not depend holomorphically on $\delta$ as $\delta\rightarrow 0$. For this reason, we will modify the coordinate changes slightly, obtaining a slightly weaker form for the local coordinates. These coordinates do depend holomorphically on $\delta$ for all $|\delta|<1$, and the Fatou coordinates for $F^\delta$ can then be introduced in the exact same way as for the map $T$. It will follow that as $\delta\rightarrow 0$ the Fatou coordinates of $F^\delta$ will degenerate to the Fatou coordinates of the one-dimensional function $f$.

\subsection{The coordinate changes}

We consider the map $F^\delta$ as in \eqref{Fdelta}, with $|\delta|<1$ and $q(z) = z^2 + O(z^3)$.

\begin{prop}\label{PropFlocal}
For any fixed integer $l>0$ there exist coordinate changes $U^\delta$, defined locally near the origin of $\mathbb{C}^2$, such that
 $\widetilde{F}^\delta = U^\delta\circ F^\delta \circ (U^\delta)^{-1}$ is of the form
\begin{align}\label{TildeFdelta}
\widetilde{F}^\delta\begin{pmatrix}z\\x \end{pmatrix} = \begin{pmatrix}z+ z^2 + a_3^\delta z^2 + \dots + a_l^\delta z^l +a_{l+1}^\delta(x)z^{l+1}+\dots\\ b_0(x) + b_1^\delta(x)z+ b_2^\delta(x)z^2 + \dots \end{pmatrix},
\end{align}
where $b_0^\delta(x)=\delta x + O(x^2)$.

Moreover, $U^\delta$ satisfies the following conditions:
\begin{enumerate}[(U1)]
\item $U^\delta$ is of the form
$$
U^\delta \left( \begin{array}{c} z \\ x \end{array}\right)= \left( \begin{array}{c} z + O(x^2, zx) \\ x \end{array}\right).
$$
In particular, $U^\delta$ is tangent to the identity.
\item $U^\delta$ depends holomorphically on $\delta$ for $|\delta|<1$.
\item As $\delta\rightarrow 0$, $U^\delta$ degenerates to the identity map, that is $U^0(z,x)=(z,x)$.
\end{enumerate}

\end{prop}

\begin{proof}
The final coordinate change $U^\delta$ will be defined as a composition of several coordinate changes, each of which we will introduce in the following steps.
\subsubsection*{Step 1}
The map $F^\delta$ has a semi-parabolic semi-attracting fixed point at $(0,0)$. It is well known that there exists an invariant strong stable manifold corresponding to the eigenvalue $\delta$ of $DF^\delta(0,0)$. This manifold is locally a holomorphic graph over the $x$-plane, depending holomorphically on $\delta$ for $|\delta|<1$ and tangent to the $x$-plane at the origin $(0,0)$, hence we can write it as $S^\delta(x)= s_2^\delta x^2 + s^\delta_3 x^3 + \dots$. The coordinate change $U_1^\delta$ will map the strong stable manifold to the $x$-plane and is defined by
\begin{align}\label{Ueda1}
U_1^\delta \left( \begin{array}{c} z \\ x \end{array}\right) := \left( \begin{array}{c} z-S^\delta(x) \\ x \end{array}\right) = \left( \begin{array}{c} z +O(x^2) \\ x \end{array}\right).
\end{align}
We see that $(U^\delta_1)^{-1}(0,x)=(S(x),x)$, hence a point in the $x$-plane is mapped to the strong stable manifold by $(U^\delta_1)^{-1}$. Since the strong stable manifold is invariant under $F^\delta$, we get that $F^\delta\circ (U^\delta_1)^{-1}(0,x) = (S^\delta(x^\prime), x^\prime)$, another point on the strong stable manifold. Applying $U_1^\delta$ we obtain for $F_1^\delta := U_1^\delta\circ F^\delta \circ (U_1^\delta)^{-1}$ that
$F_1^\delta (0,x)= (0,x^\prime).$
It follows that the $x$-plane is invariant under $F_1^\delta$, that is there are no pure $x$-terms in the first coordinate of $F_1^\delta$ and we have the form
\begin{align}\label{F1delta}
F_1^\delta \left( \begin{array}{c} z \\ x \end{array}\right)  = \left( \begin{array}{c} a_1^\delta(x) z + a^\delta_2(x) z^2 + \dots \\ b^\delta_0(x)  + b^\delta_1(x) z + b^\delta_2(x) z^2 + \dots \end{array}\right),
\end{align}
where $a_1^\delta(x) = 1 + O(x)$ and $b^\delta_0(x)=\delta x + O(x^2)$.
Since the strong stable manifold depends holomorphically on $\delta$ and is equal to the $x$-plane for $\delta=0$, this change of coordinates depends holomorphically on $\delta$ and converges to the identity as $\delta\rightarrow 0$.
In particular it follows that
\begin{align}\label{beta00}
b_0^0\equiv 0,
\end{align}
because $F_1^0=F^0$ is independent of $x$.

\subsubsection*{Step 2}
Starting with the form $F_1^\delta$ as in \eqref{F1delta}, we want to change coordinates such that $a_1^\delta(x)$ can be assumed to be constantly equal to $1$.
This can be done by by a coordinate change
$$
U_2^\delta \left( \begin{array}{c} z \\ x \end{array}\right) := \left( \begin{array}{c} P(x) z \\ x \end{array}\right),
$$
where
$$
P(x):=\prod_{n=0}^\infty a_1^\delta((b_0^\delta)^n(x)).
$$
This product is convergent since $b_0^\delta$ has an attracting fixed point at $0$ and $a_1^\delta(0)=1$.
In addition, since $P(0)=1$, we have $P(x)=1+O(x)$ and
\begin{align}\label{Ueda2}
U_2^\delta \left( \begin{array}{c} z \\ x \end{array}\right) = \left( \begin{array}{c}z + O(zx) \\ x \end{array}\right).
\end{align}
To see that this coordinate change works, we will check that $U_2^\delta\circ F_1^\delta$ is equal to $F_2^\delta\circ U_2^\delta$, where
\begin{align}\label{F2delta}
F_2^\delta \left( \begin{array}{c} z \\ x \end{array}\right)  = \left( \begin{array}{c} z + a^\delta_2(x) z^2 + \dots \\ b^\delta_0(x)  + b^\delta_1(x) z + b^\delta_2(x) z^2 + \dots \end{array}\right).
\end{align}
Here the coefficient functions $a^\delta_2, \dots, b_1^\delta, \dots$ can be different from the ones in \eqref{F1delta}. However, $b_0^\delta$ will not change, since $U_2$ does not change pure $x$-terms.
We obtain
\begin{align*}
U_2^\delta\circ F_1^\delta \left( \begin{array}{c} z \\ x \end{array}\right) &= \left( \begin{array}{c} P(b_0^\delta (x) + \dots) (a_1^\delta(x) z + \dots) \\  b_0^\delta (x) +\dots \end{array}\right) \\
&= \left( \begin{array}{c}  \left[a_1^\delta(x) \prod_{n=0}^\infty a_1^\delta((b_0^\delta)^{n+1}(x))\right] z + \dots  \\  b_0^\delta (x) +\dots \end{array}\right)
\end{align*}
and
\begin{align*}
F_2^\delta\circ U_2^\delta \left( \begin{array}{c} z \\ x \end{array}\right)  &=  \left( \begin{array}{c} P(x) z +\dots  \\ b_0^\delta(x)+\dots \end{array}\right)\\
& = \left( \begin{array}{c} \left[ \prod_{n=0}^\infty a_1^\delta((b_0^\delta)^{n}(x))\right] z + \dots  \\  b_0^\delta (x) +\dots \end{array}\right)
\end{align*}
which are equal for suitable higher order coefficient functions in \eqref{F2delta}. Hence we have $F_2^\delta = U_2^\delta \circ F_1^\delta \circ (U_2^\delta)^{-1}$.

We note that also this change of coordinates depends holomorphically on $\delta$ and that $U_2^0$ is the identity, since $P(x)\equiv 1$ for $\delta=0$ by \eqref{beta00}.

\subsubsection*{Step 3}
We prove by induction that for any fixed integer $l\geq 2$ we can choose a coordinate system with respect to which $a_1(x), a_2(x) \dots, a_l(x)$ are constants, that is we want to change coordinates such that we obtain
\begin{align}\label{F3j}
F_{3,j}^\delta \left( \begin{array}{c} z \\ x \end{array}\right)  = \left( \begin{array}{c} z + \dots + a^\delta_j z^j + a^\delta_{j+1}(x) z^{j+1} + \dots \\ b^\delta_0(x)  + b^\delta_1(x) z + b^\delta_2(x) z^2 + \dots \end{array}\right),
\end{align}
for $j=1,\dots l$ where at each step the higher order coefficient functions $a^\delta_{j+1}, \dots$ and  $b_1^\delta, \dots$ can be different from the ones before.
We see that $F_2$ as in \eqref{F2delta} is of the form $F_{3,1}^\delta$. So let us start with the form $F_{3,j}^\delta$ and show that by a coordinate change $U_{3,j+1}^\delta$ it can be brought to the form $F_{3,j+1}^\delta$.
The coordinate change is defined by
$$
U_{3,j+1}^\delta \left( \begin{array}{c} z \\ x \end{array}\right)  :=  \left( \begin{array}{c} z - P_{j+1}(x) z^{j+1} + \dots \\ x \end{array}\right),
$$
in such a way that
$$
(U_{3,j+1}^\delta)^{-1} \left( \begin{array}{c} z \\ x \end{array}\right)  =  \left( \begin{array}{c} z + P_{j+1}(x) z^{j+1} \\ x \end{array}\right),
$$
where in both cases $P_j$ is given by the convergent sum
$$
P_j(x) := \sum_{n=0}^\infty (a^\delta_j(0) - a^\delta_j((b_0^\delta)^n(x))).
$$
We see that $P_j(0)=0$, hence $P_j(x)=O(x)$ and
\begin{align}\label{Ueda3}
U_{3,j}^\delta \left( \begin{array}{c} z \\ x \end{array}\right) = \left( \begin{array}{c} z + O(z^j x) \\ x \end{array}\right) .
\end{align}
Similar to Step 2 we compare $(U^\delta_{3,j+1})^{-1}\circ F_{3,j+1}^\delta$ to $F_{3,j}^\delta \circ (U^\delta_{3,j+1})^{-1}$.
We obtain
\begin{align*}
& (U^\delta_{3,j+1})^{-1}\circ F_{3,j+1}^\delta  \left( \begin{array}{c} z \\ x \end{array}\right)  \\
= &   \left( \begin{array}{c} z + \dots +(a_{j+1}^\delta+ P_{j+1}(b_0^\delta(x) + \dots)) z^{j+1} + \dots \\ b_0^\delta(x) + \dots \end{array}\right) \\
= &  \left( \begin{array}{c} z + \dots +\left[a_{j+1}^\delta + \sum_{n=0}^\infty (a^\delta_{j+1}(0) - a^\delta_{j+1}((b_0^\delta)^{n+1}(x)))\right]z^{j+1} + \dots \\ b_0^\delta(x) + \dots \end{array}\right)
\end{align*}
and
\begin{align*}
& F_{3,j}^\delta \circ (U^\delta_{3,j+1})^{-1} \left( \begin{array}{c} z \\ x \end{array}\right) \\
= & \left( \begin{array}{c} z + \dots +(a_{j+1}^\delta(x) + P_{j+1}(x)) z^{j+1} + \dots \\ b_0^\delta(x) + \dots \end{array}\right) \\
 = & \left( \begin{array}{c} z + \dots + \left[a_{j+1}^\delta(x) + \sum_{n=0}^\infty (a^\delta_{j+1}(0) - a^\delta_{j+1}((b_0^\delta)^{n}(x)))\right] z^{j+1} + \dots  + \dots \\ b_0^\delta(x) + \dots \end{array}\right)
\end{align*}
which are equal if we take $a_{j+1}^\delta:=a_{j+1}^\delta(0)$ and suitable higher order coefficient functions for $F^\delta_{3,j+1}$.
Hence $F_{3,j+1}^\delta = U_{3,j+1}^\delta \circ F_{3,j}^\delta \circ (U_{3,j+1}^\delta)^{-1}.$

Each of the coordinate changes $U_{3,j}$ depends holomorphically on $\delta$ and converges to the identity, since $P_j(x)\equiv 0$ for $\delta=0$ by  \eqref{beta00}.

\subsubsection*{Combining Steps 1 to 3}

We define
$$
U^\delta : =  U^\delta _{3,l}\circ U^\delta_{3,l-1} \circ \dots \circ U^\delta_{3,2}\circ U_2^\delta \circ U_1^\delta.
$$
Then $U^\delta\circ F^\delta \circ (U^\delta)^{-1} = F^\delta_{3,l}$, which is in the form
$$
F^\delta_{3,l}\begin{pmatrix}z\\x \end{pmatrix} = \begin{pmatrix}z+a_2^\delta z^2 + a_3^\delta z^3 + \dots + a_l^\delta z^l +a_{l+1}^\delta(x)z^{l+1}+\dots\\ b_0(x) + b_1^\delta(x)z+ b_2^\delta(x)z^2 + \dots \end{pmatrix}.
$$
We combine equations \eqref{Ueda1}, \eqref{Ueda2} and \eqref{Ueda3} to obtain that
$$
U^\delta \left( \begin{array}{c} z \\ x \end{array}\right)= \left( \begin{array}{c} z + O(x^2, zx) \\ x \end{array}\right).
$$
Combining this with $q(z)=z^2+O(z^3)$ it also follows that $a_2^\delta=1$, independently of $\delta$, by comparing the pure $z^2$ terms of both sides of the equation $F^\delta_{3,l}\circ U^\delta = U^\delta \circ F^\delta$.
This shows that $F^\delta_{3,l}$ is in the desired form \eqref{TildeFdelta}.
In addition, since each of the coordinate changes in Steps 1 to 3 depends holomorphically on $\delta$ and converges to the identity as $\delta\rightarrow 0$, so does the composition $U^\delta$.

\end{proof}

\begin{remark}
It is worth noting that as $\delta\rightarrow 0$ the coordinate changes $U^\delta$ do not necessarily converge to the identity for a general semi-parabolic semi-attracting map with eigenvalues $1$ and $\delta$. It is the fact that the limit map $F^0$ does not depend on $x$ at all which allows us to get this result.
\end{remark}
\begin{remark}
Let us also point out the differences between Ueda's original coordinate changes and our coordinate changes. After mapping the strong stable manifold to the $x$-plane, Ueda uses Koenigs Theorem to linearize the action of the map $T$ on this plane. However, as $\delta\rightarrow 0$, the situation changes from an attracting fixed point to a superattracing one, where the linearization is generally not possible. Since we do not want to exclude the case $\delta=0$, we cannot use this change of coordinates. As a result we also made minor changes in Step 2 and 3, where Ueda used of the fact that the action on the $x$-plane had already been linearized.
Moreover, after Step 3, Ueda shows that for any fixed integer $j$ there is another change of coordinates such that the coefficient functions $b_1^\delta(x), b_2^\delta(x), \dots b_j^\delta(x)$ in \eqref{TildeFdelta} can be chosen to be linear monomials, that is $b_1^\delta(x)=b_1^\delta  x, b_2^\delta(x)=b_2^\delta  x, \dots$. Again, this change of coordinates does not depend holomorphically on $\delta$, so we do not use it.
\end{remark}

\subsection{Degeneration of Fatou coordinates}

We have seen in the previous section that the polynomial automorphism
$$
F^\delta \left( \begin{array}{c} z \\ x \end{array}\right) = \left( \begin{array}{c} z + q(z+ \delta x) \\ \delta x - q(z+ \delta x)\end{array}\right),
$$
can locally be brought to the form
\begin{align}\label{TildeFdelta2}
\widetilde{F}^\delta\begin{pmatrix}z\\x \end{pmatrix} = \begin{pmatrix}z+ z^2 + a_3^\delta z^3 + a_4^\delta(x)z^4 +\dots  \\ b_0(x) + b_1^\delta(x)z+ b_2^\delta(x)z^2 +\dots \end{pmatrix},
\end{align}
with $b_0^\delta(x)=\delta x + O(x^2)$.
In addition we have
\begin{align}\label{F0}
\pi_z\circ \widetilde{F}^0 (z,x) = \pi_z\circ F^0(z,x) = z+q(z) = f(z)
\end{align}
by the property (U3) in Proposition \ref{PropFlocal}.

The following proposition shows that the two-dimensional Fatou coordinates of $F^\delta$ degenerate holomorphically to the one-dimensional Fatou coordinates of $f$ as $\delta\rightarrow 0$.

\begin{prop}\label{PropFC}
There exist families of holomorphic maps $\Phi_{F^\delta}:\mathcal{B}_{F^\delta}\rightarrow \mathbb{C}$ and \linebreak $\Psi_{F^\delta}: \mathbb C \rightarrow \mathbb{C}^2$, defined for $|\delta| <1$, satisfying the functional equations
$$
\Phi_{F^\delta}\circ F^\delta = T_1 \circ \Phi_{F^\delta}
$$
and
$$
F^\delta \circ \Psi_{F^\delta} = \Psi_{F^\delta} \circ T_1,
$$
where $T_1(Z)=Z+1$.
The coordinates $\Phi_{F^\delta}$ and $\Psi_{F^\delta}$ depend holomorphically on $\delta$, and the degenerate coordinates $\Phi_{F^0}, \Psi_{F^0}$ coincide with the one-dimensional coordinates, that is
$$
\Phi_{F^0}(z,x)=\phi_f(z) \;\; \text{and} \; \; \pi_z \circ \Psi_{F^0}(Z)=\psi_f(Z).
$$
\end{prop}

\begin{proof}

Two-dimensional Fatou coordinates have been constructed by Ueda in \cite{Ueda1,Ueda2}.

He starts with a map in the local form
$$
T\begin{pmatrix}z\\x \end{pmatrix} = \begin{pmatrix}z+a_2 z^2 + a_3 z^3 + O(z^4) \\ \delta x + b_1 x z+ O(z^2) \end{pmatrix}.
$$
In order to construct Fatou coordinates, the idea is to introduce a further change of coordinates which leaves $x$ invariant and maps $z$ to
$$
Z:=\omega^{\iota/o}(z):= -\frac{1}{z}-b\log\left(\mp\frac{1}{z}\right) \footnote{This notation is adapted from \cite{BSU}, where the authors use $w^{\iota / o},$ the superscripts $\iota$ and $o$ refering to incoming and outgoing, respectively. Since $w$ is one of the coordinates of our four dimensional mapping, we choose to avoid this letter but use $\omega$ instead.}
$$
where $b=1-a_3$ and we use $\omega^\iota$ on the attracting and $\omega^o$ on the repelling side. In these coordinates $T$ is locally close to a translation by $1$ on the $Z$ coordinate:
\begin{align}\label{Zplus1}
T: Z\mapsto Z+1+O\left(\frac{1}{Z^2}\right),
\end{align}
independently of $x$.
Note that the $\log$-term in $\omega^{\iota/o}$ and the choice of $b=1-a_3$ are used to cancel the terms of order $1/Z$ in \eqref{Zplus1}.

The local attracting Fatou coordinate can then be constructed as the limit
\begin{align}\label{AttrFCDef}
\Phi_{T}(z,x) := \lim_{n\rightarrow\infty}( \omega^\iota \circ \pi_z \circ T^n(z,x) - n)
\end{align}
and in \cite{Ueda2} the repelling Fatou coordinate is constructed as the limit
 \begin{align}
 \Psi_{T}(Z):=\lim_{n\rightarrow\infty}T^n((\omega^o)^{-1} (Z-n) , 0).
\end{align}
The proof that these limits exist relies on \eqref{Zplus1}: the translation by one per iteration cancels with the subtraction of $n$ in \eqref{AttrFCDef} and the infinite sum of remainder terms $O(1/Z^2)$ converges, because $\pi_z\circ T^n(Z, \cdot)\sim n$ as $n\rightarrow\infty$. The situation for the repelling case is slightly more involved, but the idea is the same. In order to obtain the form \eqref{Zplus1} after the coordinate change $\omega^{\iota/o}$, it is sufficient that the $z$-coordinate of $T$ is of the form $z+a_2 z^2 + a_3 z^3 + O(z^4)$. Here $a_2$ and $a_3$ are independend of $x$, but the terms in $O(z^4)$ can initially depend on $x$. However, in the parabolic basin the $x$-coordinate becomes and remains bounded by a constant $\eta>0$ under iteration of $T$. These properties are satisfied by our local form \eqref{TildeFdelta2} as well, and hence its Fatou coordinates can be constructed in the exact same way, that is
$$
\Phi_{\widetilde{F}^\delta}(z,x) := \lim_{n\rightarrow\infty}( \omega^\iota \circ \pi_z \circ (\widetilde{F}^\delta)^n(z,x) - n)
$$
and
 $$
 \Psi_{\widetilde{F}^\delta}(Z):=\lim_{n\rightarrow\infty}(\widetilde{F}^\delta)^n((\omega^o)^{-1} (Z-n) , 0).
 $$
As the local map $\widetilde{F}^\delta$ varies holomorphically with $\delta$, so do the attracting and repelling Fatou coordinates.
Also note that the construction of one-dimensional Fatou coordinates can be seen as a special case of these constructions. Indeed, for a one-dimensional function $f$ the Fatou coordinates $\varphi_f$ and $\psi_f$ can be constructed as
$$
\varphi_{f}(z) := \lim_{n\rightarrow\infty}( \omega^\iota \circ f^n(z) - n)
$$
and
 $$
 \psi_{f}(Z):=\lim_{n\rightarrow\infty}f^n((\omega^o)^{-1} (Z-n)).
 $$
It follows from \eqref{F0} that
\begin{align}\label{Phi0}
\Phi_{\widetilde{F}^0}(z,x)=\phi_f(z)
\end{align}
and
\begin{align}\label{Psi0}
 \pi_z\circ \Psi_{\widetilde{F}^0}(Z)=\psi_f(Z).
\end{align}

Going back from the local form $\widetilde{F}^\delta$ to the original map $F^\delta$ we obtain attracting and repelling Fatou coordinates
$$\
\Phi_{F^\delta}=\Phi_{\widetilde{F}^\delta}\circ U^\delta \qquad \text{and} \qquad \Psi_{F^\delta}= (U^\delta)^{-1} \circ \Psi_{\widetilde{F}^\delta}
$$
for $F^\delta$ and these can be extended to the attracting basin and to $\mathbb{C}$, respectively, by using the functional equations.
Since both the local Fatou coordinates $\Phi_{\widetilde{F}^\delta}$, $\Psi_{\widetilde{F}^\delta}$ and our coordinate changes $U^\delta$ depend holomorphically on $\delta$ for $|\delta|<1$, so do the Fatou coordinates $\Phi_{F^\delta}$ and $\Psi_{F^\delta}$. Moreover, since $U^0$ is the identity, equations \eqref{Phi0} and \eqref{Psi0} also hold in the original coordinates where $\tilde{F^0}$ is replaced by $F^0$.
\end{proof}

\subsection{Proof of Proposition B'}

By Proposition \ref{PropFC} the maps
$$
\Psi_{F^\delta} \circ \Phi_{F^\delta} : \mathcal{B}_{F^\delta} \rightarrow \mathbb C^2
$$
vary holomorphically with $\delta$ and degenerate to $\psi_f\circ \phi_f$ as $\delta \rightarrow 0$. It also follows that $\mathcal{B}_{F^0} = \mathcal{B}_f \times \mathbb C$. Since $\psi_f\circ \phi_f$ is assumed to have an attracting fixed point $\hat{z}_0$, it follows that
$$
\pi_z \circ \Psi_{F^0} \circ \Phi_{F^0}(\hat{z}_0,x) = \hat{z}_0.
$$
Since $\Phi_{F^0}(z,x)$ does not depend on $x$, there exists $\hat{x}_0$ for which
$$
\Psi_{F^0} \circ \Phi_{F^0}(\hat{z}_0, \hat{x}_0) = (\hat{z}_0, \hat{x}_0).
$$
Since $\hat{z}_0$ is attracting we can choose $r>0$ sufficiently small such that the cylinder $D(\hat{z}_0,r)\times \mathbb C$ is mapped strictly inside itself by $\Psi_{F^0} \circ \Phi_{F^0}$. Since $\Phi_{F^0}$ is independent of $x$, it follows that for $r^\prime$ sufficiently large $\Psi_{F^0} \circ \Phi_{F^0}$ maps the polydisk \linebreak $D(\hat{z}_0,r)\times  D(\hat{x}_0, r')$ relatively compactly inside itself, hence $\Psi_{F^0} \circ \Phi_{F^0}$ contracts the Kobayashi metric uniformly. It follows that all orbits in $D(\hat{z}_0,r)\times D(\hat{x}_0, r')$ converge to a unique attracting fixed point, which must equal $(\hat{z}_0, \hat{x}_0)$. The fact that $\Psi_{F^\delta} \circ \Phi_{F^\delta}$ depends continuously on $\delta$ implies the existence of attracting fixed points $(\hat{z}_\delta, \hat{x}_\delta)$ for $|\delta|$ sufficiently small. \hfill $\square$

\section{Approximate Fatou coordinates}

In this section we will introduce approximate Fatou coordinates in dimension two. They will be defined exactly as in \cite{ABDPR} for the one-dimensional case. In the first subsection we recall definitions and results thereof.

\subsection{Notation and one-dimensional results}

Define the functions
$$
f_w(z)=f(z)+\frac{\pi^2}{4}w=z+z^2+az^3+O(z^4) + \frac{\pi^2}{4}w
$$
and
$$
g(w)= w - w^2 +O(w^3).
$$
The ``approximate Fatou coordinates'' will be coordinate changes $\phi_w$ such that \linebreak $\phi_{g(w)}\circ f_w \circ \phi_w^{-1}$ is close to a translation.
We start by choosing $r>0$ small enough such that $B_r:=D(r,r) \subset \mathcal{B}_g$ and $g(B_r)\subset B_r$.
We assume that $w\in B_r$ and hence $g^m(w)\rightarrow 0$ as $m\rightarrow\infty$.
By $\sqrt{w}$ we denote the principle value of the square root on $D(r,r)$.
As in \cite{ABDPR}, we fix a real number
$$
\frac{1}{2}<\alpha<\frac{2}{3}
$$
and define for $w\in B_r$
$$
r_w:= |w|^{(1-\alpha)/2} \qquad \text{and} \qquad R_w:= |w|^{-\alpha/2}.
$$
Denote by $\mathcal{R}_w$ the rectangle
$$
\mathcal{R}_w:=\left\{\mathcal{Z}\in\mathbb{C}: \frac{r_w}{10}<\Re(\mathcal Z)<1-\frac{r_w}{10}~\text{and}~-\frac{1}{2}<\Im(\mathcal Z)<\frac{1}{2}\right\},
$$
and let $D^{att}_w$ and $D^{rep}_w$ be the disks
$$
D^{att}_w:=D(R_w, R_w/10) \qquad \text{and} \qquad D^{rep}_w:=D(-R_w,R_w/10).
$$
In \cite{ABDPR}, the question was raised whether there exist invariant parabolic curves $\zeta^\pm:B_r\rightarrow \mathbb{C}$ such that $\zeta^\pm(w)\rightarrow 0$ as $w\rightarrow 0$ and such that $f_w\circ\zeta^\pm(w)=\zeta^\pm\circ g(w)$.
Quite recently in \cite{LHR} the existence of such parabolic curves $\zeta^\pm$ has been proved.
For our work here it will not matter whether we define $\zeta^\pm$ to be these invariant parabolic curves or the approximate invariant curves defined in \cite{ABDPR}.
All we require is that they satisfy
$$
\zeta^\pm(w)\sim \pm\frac{\pi i}{2} \sqrt{w} + O(w).
$$

As in \cite{ABDPR}, let $\psi_w:\mathbb{C}\rightarrow\mathbb{P}^1(\mathbb{C})\backslash\{\zeta^+(w),\zeta^-(w)\}$ denote the universal cover given by
$$\psi_w(\mathcal{Z}):=\frac{\zeta^-(w)\cdot e^{2\pi i \mathcal{Z}}-\zeta^+(w)}{e^{2\pi i \mathcal{Z}} -1} = -\frac{\pi}{2} \sqrt{w} \cot(\pi\mathcal{Z}) + O(w).$$
The restriction of this function to the strip
$$
\mathcal{S}_0:=\{\mathcal{Z}\in\mathbb{C}: 0<\Re(Z)<1\}
$$
is univalent, and the inverse is given by
$$
\psi_w^{-1}(z):= \frac{1}{2\pi i} \log \left(\frac{z-\zeta^+(w)}{z-\zeta^-(w)}\right),
$$
where $\log(\cdot)$ is the branch of the logarithm defined on $\mathbb{C}\backslash\mathbb{R}^+$ for which $\log(-1)=\pi i$.
Let $\chi_w:\mathcal{S}_0\rightarrow \mathbb{C}$ be defined by
$$\chi_w(\mathcal{Z}):=\mathcal{Z} - \frac{\sqrt{w}(1-a)}{2}\log\left(\frac{2\sin(\pi\mathcal{Z})}{\pi\sqrt{w}}\right),$$
where this time $\log(\cdot)$ is the branch of the logarithm defined on $\frac{1}{\sqrt{w}}(\mathbb{C}\backslash\mathbb{R}^-)$ for which $\log(1)=0$.
Define the set
$$\mathcal{S}_w:=\{\mathcal{Z}\in\mathbb{C}: |w|^{1/4}<\Re(\mathcal{Z})<1-|w|^{1/4}\}$$
and its image under $\psi_w$:
$$V_w:=\psi_w(\mathcal{S}_w).$$
With these definitions the approximate Fatou coordinates were introduced in \cite{ABDPR} as follows:
\begin{definition}
The approximate Fatou coordinates $\varphi_w$ are the maps
$$\varphi_w:= \chi_w\circ\psi_w^{-1}: V_w \rightarrow \mathbb{C}, \qquad w\in B_r.$$
\end{definition}
A large amount of the work in \cite{ABDPR} consists of proving that the approximate Fatou coordinates satisfy the following three properties.
\begin{prop1}
As $w\rightarrow 0$ in $B_r$,
$$
D^{att}_w \subset \phi_f(V_w\cap P^{att}_f) \qquad \text{and} \qquad \sup_{Z\in D^{att}_w} \left|\frac{2}{\sqrt{w}} \varphi_w \circ \phi_f^{-1}(Z) - Z \right| \rightarrow 0.
$$
\end{prop1}
\begin{prop2}
As $w\rightarrow 0$ in $B_r$,
$$
1+\frac{\sqrt{w}}{2} D^{rep}_w \subset \varphi_w(V_w\cap P^{rep}_f)
$$
and
$$
\sup_{Z\in D^{rep}_w} \left|\psi_f^{-1}\circ\varphi_w^{-1}\left(1+\frac{\sqrt{w}}{2} Z\right) - Z \right| \rightarrow 0.
$$
\end{prop2}
Here $P^{att}_f$ and $P^{rep}_f$ are small sets close to the origin on the attracting and repelling side, respectively. We will define them precisely together with their two dimensional equivalents in the next section.
Properties 1 and 2 assert that $\frac{2}{\sqrt{w}}\varphi_w$ is in a sense close both to the attracting and to the repelling Fatou coordinate.
The third property shows that the action of $f_w$ in the approximate Fatou coordinates is close to a translation.
\begin{prop3}
As $w\rightarrow 0$ in $B_r$,
$$
\mathcal{R}_w\subset\varphi_w(V_w), \qquad f_w\circ\varphi_w^{-1}(\mathcal{R}_w)\subset V_{g(w)}
$$
and
$$
\sup_{\mathcal{Z}\in\mathcal{R}_w} \left| \varphi_{g(w)}\circ f_w \circ \varphi_w^{-1}(\mathcal{Z}) - \mathcal{Z} - \frac{\sqrt{w}}{2}\right| = o(w).
$$
\end{prop3}
These properties make up the main part of the proof of Proposition A.
We note that they hold for any functions $f$ and $g$ of the forms $f(z)=z+z^2+a z^3 + O(z^4)$ and $g(w)=w-w^2+O(w^3)$. In particular we can take $a=a^\delta_3$, the coefficient in \eqref{TildeFdelta2}, and we will implicitly use this choice throughout this section.

\subsection{Two-dimensional setting}

We will introduce the $2$-dimensional analogues of the statements discussed in the previous subsection. Instead of repeating the proofs of Properties 1, 2 and 3 we will often be able to estimate the difference between one- and two-dimensional maps in order to deduce our results from the one-dimensional analogs.

We consider the maps
$$
F\left( \begin{array}{c} z \\ x \end{array}\right) = \left( \begin{array}{c} z + q_1(z+ \delta x) \\ \delta x - q_1(z+ \delta x)\end{array}\right) \; \;
\mathrm{and} \; \;
G\left( \begin{array}{c} w \\ y \end{array}\right) = \left( \begin{array}{c} w + q_2(w+ \delta y)\\ \delta y - q_2(w+ \delta y)\end{array}\right),
$$
with $q_1(z)=z^2+az^3$ and $q_2(w)=-w^2+O(w^3)$.
This time, we denote by $D_r$ the polydisk $D_r:=D(r,r)\times D(0,r)$ with $r$ small enough that $G(D_r)\subset D_r$. We take $(w,y)\in D_r$ so that $G^m(w,y)\rightarrow(0,0)$ as $m\rightarrow \infty$.
We will use the local coordinates introduced earlier:
\begin{align}\label{Flocal}
\widetilde{F} \left( \begin{array}{c} z \\ x \end{array}\right) := U\circ F\circ U^{-1} \left( \begin{array}{c} z \\ x \end{array}\right)=  \left( \begin{array}{c} z+z^2+a_3 z^3 + O(z^4) +\dots\\ b_0(x) + b_1(x) z + b_2(x) z^2 + \dots \end{array}\right),
\end{align}
where $b_0(x)=\delta x + O(x^2)$ as in equation \eqref{TildeFdelta2}.
Note that we have dropped the superscript $\delta$ for both $\widetilde{F}$ and $F$ as opposed to earlier sections, because we now take a fixed $\delta$, with $0<|\delta|<\delta_0$ small enough. The condition that $\delta$ is small will be used in Lemma \ref{LemmaOmega}.

Let us introduce some further notation.
 The attracting petal $P^{att}_{\widetilde{F}}$ is defined by
$$
P^{att}_{\widetilde{F}}:= \{(z,x)\in\mathbb{C}^2: \Re\left(-\frac{1}{z}\right)>R, |x|<\eta\},
$$
where $R>0$ is large and $\eta>0$ small. In particular we choose these constants so that the coordinate changes from the previous section and the map $\widetilde{F}$ are defined on $P^{att}_{\widetilde{F}}$.
We define the one-dimensional attracting petal for $f$ to be the projection of $P^{att}_{\widetilde{F}}$ to the $z$-plane, that is
$$
P^{att}_f:=\{z\in\mathbb{C}: \Re\left(-\frac{1}{z}\right)>R\}.
$$
Finally, we define the attracting petal $P^{att}_F$ for the original map $F$ by
$$
P^{att}_F := U^{-1}(P^{att}_{\widetilde{F}}).
$$
Note that in contrast to the one-dimensional case the attracting Fatou coordinate $\Phi_F$ is not injective, as a map from the attracting petal $P^{att}_F\subset\mathbb{C}^2$ into $\mathbb{C}$.
The repelling petals are defined as the images of left-half planes under the respective repelling Fatou coordinate, that is
$$
P^{rep}_f:=\psi_f(\{Z\in\mathbb{C}: \Re(Z)<-R^\prime\})\subset\mathbb{C},
$$
and
$$
 P^{rep}_{\widetilde{F}}:=\Psi_{\widetilde{F}}(\{Z\in\mathbb{C}: \Re(Z)<-R\})\subset\mathbb{C}^2,
$$
where we choose $R^\prime$ large enough compared to $R$ so that $P^{rep}_f \subset \pi_z(P^{rep}_{\widetilde{F}})$. We remark that this is possible by equation \eqref{RepFatCoordProp} below, and its one-dimensional equivalent.

Both the one- and two-dimensional repelling Fatou coordinates are injective on the above half-planes when $R$ and $R^\prime$ are large enough.
However, note that $P^{rep}_{\widetilde{F}}$ is a one-dimensional subset of $\mathbb{C}^2$, in fact it is locally a graph over the $z$-plane.

We define the repelling petal $P^{rep}_F$ for the original map $F$ by
$$
P^{rep}_F := U^{-1}(P^{rep}_{\widetilde{F}}).
$$
We choose $R, R^\prime$ large enough and $\eta$ small enough such that all attracting petals are forward invariant and all repelling petals are backward invariant under iteration of the corresponding maps.

We note that the two-dimensional attracting and repelling Fatou coordinates for $\widetilde{F}^\delta$ satisfy similar estimates as in one dimension (see also \cite{BSU}).  As $P^{att}_{\widetilde{F}^\delta}\ni(z,x) \rightarrow 0$ we have
\begin{align}\label{AttrFatCoordProp}
\Phi_{\widetilde{F}^\delta}(z,x)=-\frac{1}{z}-b\log\left(-\frac{1}{z}\right)+o(1).
\end{align}
For the repelling coordinate we need to be careful, since $\Psi_{\widetilde{F}^\delta}^{-1}$ is only defined on the one-dimensional set $P^{rep}_{\widetilde{F}^\delta}$. We obtain that for $P^{rep}_{\widetilde{F}^\delta} \ni (z,x) \rightarrow 0$ we have
\begin{align}\label{RepFatCoordProp}
(\Psi_{\widetilde{F}^\delta})^{-1}(z,x) = -\frac{1}{z}-b\log\left(\frac{1}{z}\right)+o(1).
\end{align}

Let us define the perturbed polynomial $F_w:\mathbb{C}^2\rightarrow\mathbb{C}^2$ by
\begin{align}\label{Fw}
F_w \left( \begin{array}{c} z \\ x \end{array}\right) = F\left( \begin{array}{c} z \\ x \end{array}\right) + \left(\begin{array}{c} \frac{\pi^2}{4}w \\ 0 \end{array}\right)=  \left( \begin{array}{c} z + q_1(z+ \delta x) \\ \delta x - q_1(z+ \delta x)\end{array}\right) + \left( \begin{array}{c} \frac{\pi^2}{4} w \\ 0 \end{array}\right),
\end{align}
The following lemma shows how the perturbation $F_w$ of $F$ translates into a perturbation $\widetilde{F}_w$ of the local form $\widetilde{F}$.
\begin{lemma}\label{LemmaFwlocal}
Let $F_w$ be the perturbed Henon map as in \eqref{Fw}, and let $U$ be the local coordinate change that brings $F_0$ to the form $\widetilde{F}$ as in \eqref{Flocal}. Then
\begin{align}\label{TildeFw}
\widetilde{F}_w:=U \circ F_w \circ U^{-1}\begin{pmatrix}z\\x \end{pmatrix} = \widetilde{F}\begin{pmatrix}z\\x \end{pmatrix} + \begin{pmatrix}\frac{\pi^2}{4}w+ \mathcal{O}(wx, wz^2)\\0 \end{pmatrix}.
\end{align}
\end{lemma}
\begin{proof}
One easily computes
$$
F_w \circ U^{-1}  \left( \begin{array}{c} z \\ x \end{array}\right) = F \circ U^{-1}\left( \begin{array}{c} z \\ x \end{array}\right) + \left( \begin{array}{c} \frac{\pi^2}{4}w \\ 0 \end{array}\right) = \left(\begin{array}{c} z + \frac{\pi^2}{4} w\\ \delta x \end{array}\right) + O(z^2, zx, x^2).
$$
By the property (U1) of Ueda's change of coordinates, we have
$$
U \left( \begin{array}{c} z \\ x \end{array}\right)  = \left( \begin{array}{c} z + O(x^2, zx) \\ x \end{array}\right).
$$
Combining these two equations for $\widetilde{F}_w= U \circ F_w \circ U^{-1}$ gives the desired result.
\end{proof}

In order to estimate $x$ in terms of $z$, we introduce a family of domains $\Omega_w$ that are forward invariant in a neighborhood of the origin.

\begin{lemma}\label{LemmaOmega}
Let $\widetilde{F}_w$ be a map as in \eqref{TildeFw} with $\delta<\delta_0$ small. Let $w_n=\pi_w\circ G^n(w,y)$ for $(w,y)\in D_r$.
Let $\Delta^2_s$ be the polydisk in $\mathbb{C}^2$ of radius $s$ with center at the origin.
There exist constants $C>1$ and $s>0$, such that for the sets $\Omega_{w_n}$ defined by
$$
\Omega_{w_n}:=\{(z,x)\in\mathbb{C}^2: |x|<C\max\{|z|^2, |w_n|^2\}\}
$$
and $n$ large enough we have
$$
\widetilde{F}_{w_n}(\Omega_{w_n}\cap \Delta^2_s) \subseteq \Omega_{w_{n+1}}.
$$
\end{lemma}

\begin{proof}

Let $(z^\prime,x^\prime):=\widetilde{F}_{w_n}(z,x)$ for $(z,x)\in\Omega_{w_n}\cap \Delta^2_s$. By \eqref{Flocal} and \eqref{TildeFw}, we have
$$
z^\prime= z + \frac{\pi^2}{4}w_n + O (z^2, w_nx, w_n z^2) =  z + \frac{\pi^2}{4}w_n + O (z^2, w_nx)
$$
and
$$
x^\prime= \delta x + O(z^2, zx, x^2).
$$
Hence for $(z,x)\in\Delta^2_s$ with $s>0$ small enough and for $w_n$ small enough, there exists a constant $M$ such that
$$
|z^\prime|\geq |z| - \frac{\pi^2}{4}|w_n| - M (|z|^2 + |w_n x|)
$$
and
$$
|x^\prime|\leq |\delta||x|+M(|z|^2+|zx|+|x|^2).
$$
If necessary, we can choose $s$ even smaller so that
\begin{align}\label{eqn1571}
|z|<\frac{1}{16M}\qquad \text{and}\qquad |x|<\frac{1}{16M}.
\end{align}
Take $0<\delta_0<\frac{1}{4\pi^4}$ and $C>0$ very large, so that for $|\delta|<\delta_0$
\begin{align*}
\max\left\{\left(4|\delta|+\frac{1}{2}\right)C + 4M,~ 4|\delta|\pi^4C+12M\pi^4\right\}<C.
\end{align*}
Finally, for $n$ big enough, we have
\begin{align}\label{eqn1572}
|w_n|<\frac{1}{\pi^2 C}
\end{align}
and
\begin{align}\label{eqn1573}
|w_n|\leq2 |w_{n+1}|.
\end{align}
Now we distinguish two cases.\\

First, assume $\pi^2|w_n|<|z|$.
Then $|x|<C|z|^2$ by assumption and combining with \eqref{eqn1571} and  \eqref{eqn1572} we see that
$$
|z^\prime| \geq |z| - \frac{\pi^2}{4}|w_n| - M (|z^2| + C|w_n||z|^2) \geq |z|-|z|/4 - |z|/4 = |z|/2,
$$
hence
\begin{align*}
|x^\prime|&\leq \left(|\delta|+ \frac{1}{16}\right)|x| + M (|z|^2+|zx|)\\
&\leq  \left(|\delta|+\frac{1}{16}\right)C|z|^2+M (|z|^2+C|z|^3)\\
&\leq\left(\left(|\delta|+\frac{1}{8}\right)C+M\right)|z|^2\\
&\leq\left(\left(4|\delta|+\frac{1}{2}\right)C + 4M\right)|z^\prime|^2\\
&<C|z^\prime|^2,
\end{align*}
where we have used the estimate on $x$ from \eqref{eqn1571} for the first line and the estimate on $z$ from \eqref{eqn1571} to get from the second to the third line.\\

On the other hand, if $|z|\leq\pi^2 |w_n|$, then $|x|<\pi^4C|w_n|^2$ and hence
\begin{align*}
|x^\prime|&\leq  |\delta|\pi^4C|w_n|^2+M (\pi^4|w_n|^2+ \pi^6C|w_n|^3+\pi^8C^2|w_n|^4)\\
&\leq(|\delta|\pi^4 C+ 3M\pi^4)|w_n|^2\\
&\leq(4|\delta|\pi^4C+12M\pi^4)|w_{n+1}|^2\\
&<C|w_{n+1}|^2.
\end{align*}
This time we have used \eqref{eqn1572} to get from the first to the second line and \eqref{eqn1573} from the second to the third line.\\

In both cases $(z^\prime, x^\prime)\in\Omega_{w_{n+1}}$ which finishes the proof.
\end{proof}

We prove a similar but simpler statement for the $\left(\begin{array}{c} w\\y \end{array} \right)$-coordinates.

\begin{lemma}\label{LemmaOmegaG}
Let $G$ be the map
$$G\left( \begin{array}{c} w \\ y \end{array}\right) = \left( \begin{array}{c} w + q_2(w+ \delta y)\\ \delta y - q_2(w+ \delta y)\end{array}\right),$$
where $q_2(w)=-w^2+O(w^3)$ and $|\delta|<\delta_0$ small enough. Given a compact $C_G \Subset \mathcal{B}_G$ there exists a constant $C>0$ such that
\begin{align}\label{ylesswsquared}
|y_n|\leq C|w_n^2|,
\end{align}
for $n$ large enough, where $(w_n, y_n)=G^n(w,y)$ for $(w,y)\in C_G\Subset\mathcal{B}_G$.
\end{lemma}

\begin{proof}
Using the local form $\widetilde{G}=U\circ G\circ U^{-1}$, we can repeat the proof of Lemma \ref{LemmaOmega} replacing $(z,x)$ by $(w,y)$ and omitting the perturbations that the additive term $\frac{\pi^2}{4}w$ has previously inflicted upon $\widetilde{F}$. This gives a domain $\widetilde{\Omega}$, forward invariant under $\widetilde{G}$ (within a small polydisk centered at the origin), on which $\widetilde{y}=O(\widetilde{w}^2)$ for $(\widetilde{w}, \widetilde{y})=U(w,y)$.

Since $U$ has linear part equal to the identity, going back to the original coordinates, we obtain a domain $\Omega=U^{-1}(\widetilde{\Omega})$ on which  we have $|y|< C|w|^2$ and $G(\Omega\cap \Delta^2_s)\subset \Omega$ for some $s>0$.

Since the iterates of $G$ converge to $(0,0)$, we have $G^{n}(C_G)\in\Delta^2_s$ for all $n\geq n_0$ large enough. Since $G^{n_0}(C_G)$ is compact and bounded away from $(0,0)$, there is a constant $C^\prime$ such that $|y|<C^\prime|w|^2$ on $G^{n_0}(C_G)$. If necessary we increase the constant $C$ in the definition of $\Omega$ to be at least as large as $C^\prime$. Then $(w_{n_0}, y_{n_0})\in\Omega\cap\Delta^2_s$ and the result follows.
\end{proof}

We define the function $\vartheta_w: (\mathbb{C}\backslash\{\zeta^+(w),\zeta^-(w)\})\times\mathbb{C}\rightarrow\mathbb{C}$ by
$$
\vartheta_w(z,x) := \frac{1}{2\pi i} \log \left(\frac{z-\zeta^+(w)}{z-\zeta^-(w)}\right).
$$
We note that the restriction of $\vartheta_w$ to the $z$-plane is exactly the function $\psi^{-1}_w$ from \cite{ABDPR}. We want to use a different notation, since $\vartheta^{-1}_w$, which would correspond to $\psi_w$, is not well-defined. However, since $\vartheta_w$ only depends on the $z$-coordinate, there is a well-defined z-coordinate of $\vartheta_w^{-1}$, that is we may define an inverse map by $\vartheta_w^{-1}(z)=(\psi_w(z), \gamma(z))$ for any given graph $\gamma(z)$ over the $z$-plane.

\begin{definition}
The approximate Fatou coordinates $\Phi_w$ are the maps
$$\Phi_w:= \chi_w\circ\vartheta_w: V_w\times\mathbb{C} \rightarrow \mathbb{C}.$$
\end{definition}

We note that the approximate Fatou coordinate $\Phi_w$ should not be confused with $\Phi_{\widetilde{F}}$ which is the (non-approximate) attracting Fatou coordinate of $F$.

Now we can state the two-dimensional versions of Properties 1, 2 and 3. The formulation of Property 1 is slightly different from the one-dimensional analogue, since the inverse of the attracting Fatou coordinate $\Phi_{\widetilde{F}}$ is not well-defined. Instead we will consider points in the preimage of the sets $D^{att}_w$ under this map.
\begin{prop12}
Within the set $\Omega_w\cap P^{att}_{\widetilde{F}}$ we have
$$
\Phi_{\widetilde{F}}^{-1}(D^{att}_w)\subset V_w\times\mathbb{C}
$$
and moreover
$$
 \sup_{(z,x)\in \Phi_{\widetilde{F}}^{-1}(D^{att}_w)} \left|\frac{2}{\sqrt{w}} \Phi_w (z,x) - \Phi_{\widetilde{F}}(z,x) \right| \rightarrow 0,
$$
as $(w,y)\rightarrow (0,0)$ in $D_r$.
\end{prop12}
\begin{prop22}
As $(w,y)\rightarrow (0,0)$ in $D_r$,
$$
1+\frac{\sqrt{w}}{2} D^{rep}_w \subset \Phi_w((V_w\times\mathbb{C}) \cap P^{rep}_{\widetilde{F}})
$$
and
$$
\sup_{Z\in D^{rep}_w} \left|\Psi_{\widetilde{F}}^{-1}\circ\Phi_w^{-1}\left(1+\frac{\sqrt{w}}{2} Z\right) - Z \right| \rightarrow 0.
$$
\end{prop22}
Here we note that $\Phi_w^{-1}(Z)$ is a-priori not well-defined. However, it has a well-defined $z$-coordinate given by the one-dimensional inverse $z=\varphi_w^{-1}(Z)$. Using the first assertion of the statement, we can add a well-defined $x$-coordinate by choosing it in such a way that $\Phi_w^{-1}(Z)=(z,x)\in P^{rep}_{\widetilde{F}}$ for $Z\in(1+\frac{\sqrt{w}}{2} D^{rep}_w)$ (recall that locally $P^{rep}_{\widetilde{F}}$ is a graph over the $z$-plane). With this definition also $\Psi_{\widetilde{F}}^{-1}$ will be well-defined on $\Phi_w^{-1}(1+\frac{\sqrt{w}}{2} D^{rep}_w)$.
\begin{prop32}
Let $w_m=\pi_w\circ {G}^m(w,y)$, such that $w_m\rightarrow 0 $ as $m\rightarrow\infty$. Then
$$
\mathcal{R}_{w_m}\subset\Phi_{w_m}(V_{w_m}\times\mathbb{C}), \qquad \widetilde{F}_{w_m}(\Phi_{w_m}^{-1}(\mathcal{R}_{w_m}) \cap \Omega_{w_m})\subset (V_{w_{m+1}}\times\mathbb{C})\cap \Omega_{w_{m+1}}
$$
and
$$
\sup_{(z,x)\in\Phi_{w_m}^{-1}(\mathcal{R}_{w_m})\cap\Omega_{w_m}} \left| \Phi_{w_{m+1}}\circ \widetilde{F}_{w_m} (z,x) - \Phi_{w_m}(z,x) - \frac{\sqrt{w_m}}{2}\right| = o(w_m),
$$
as $w_m\rightarrow 0$.
\end{prop32}

\subsection{Error estimates in dimension two}

This section is devoted to the proofs of Properties 1', 2' and 3'.
While Properties 1' and 2' follow quickly from the one-dimensional results, more precise error estimates are needed for Property 3'.

\medskip
\noindent {\bf Proof of Property 1':}

For $\mathcal{Z}\in\mathcal{S}_w$ we have
$$
\cot(\pi\mathcal{Z})=O\left(\frac{1}{|w|^{\frac{1}{4}}}\right).
$$
Since $\psi_w(\mathcal{Z})\sim -\frac{\pi}{2} \sqrt{w} \cot(\pi\mathcal{Z}) + O(w)$, we infer for $z\in V_w=\psi_w(\mathcal{S}_w)$ that
$|z|=O(|w|^{\frac{1}{4}}).$
In particular
$$
\sup_{z\in V_w} |z| \rightarrow 0 \qquad \text{as}~ w\rightarrow 0 .
$$
It follows that $(z,x)\in(V_w\times\mathbb{C}) \cap \Omega_w \cap P^{att}_{\widetilde{F}}$ converges to $(0,0)$ as $w\rightarrow 0$.
In this case we get from \eqref{AttrFatCoordProp} that
$$\Phi_{\widetilde{F}}(z,x)=-\frac{1}{z}-b\log\left(-\frac{1}{z}\right)+o(1).$$
From the corresponding statement for the one-dimensional attracting Fatou coordinate we conclude that
\begin{align}\label{AttrFCclose}
\Phi_{\widetilde{F}}(z,x) - \phi_f (z) = o(1), \qquad \text{as}~w\rightarrow 0.
\end{align}
Here our choice of $f(z)=z+z^2+az^3+O(z^4)$ with $a=a^\delta_3$ is important in order to have the same $b=1-a$ in both the one-dimensional and the two-dimensional statement.

We note that the one-dimensional Property 1 holds for the strictly larger disk \linebreak $\mathbf{D}^{att}_w=D(R_w, R_w/9)$ instead of $D^{att}_w=D(R_w, R_w/10)$, so that
$$
\mathbf{D}^{att}_w \subset \phi_f(V_w\cap P^{att}_f)
$$
still holds. It follows from \eqref{AttrFCclose} that $\Phi_{\widetilde{F}}((V_w\times\mathbb{C}) \cap \Omega_w \cap P^{att}_{\widetilde{F}})$ approximates $\phi_f(V_w\cap P^{att}_f)$. Since $D^{att}_w\subset \mathbf{D}^{att}_w$ and the boundaries have a distance of order $R_w\rightarrow \infty$ as $w\rightarrow 0$ it follows that $D^{att}_w$ is compactly contained in $\Phi_{\widetilde{F}}((V_w\times\mathbb{C}) \cap \Omega_w \cap P^{att}_{\widetilde{F}})$.
Since each level set $\{\Phi_{\widetilde{F}} = \mathcal{Z}\}$ intersects $P^{att}_F\cap\Omega_w$ in at most one component, a nearly vertical holomorphic disk, it follows that
$$\Phi_{\widetilde{F}}^{-1}(D^{att}_w)\subset V_w\times\mathbb{C} $$
in $P^{att}_{\widetilde{F}}\cap\Omega_w$.
The second assertion follows from \eqref{AttrFCclose} by noticing that the second assertion of the one-dimensional Property 1 is equivalent to
$$
\sup_{z\in\phi_f^{-1}(D^{att}_w)} \left|\frac{2}{\sqrt{w}} \varphi_w (z) - \phi_f(z) \right| \rightarrow 0.
$$
and that by construction we have $\varphi_w(z)=\Phi_w(z,x)$.

\hfill $\square$

\medskip
\noindent {\bf Proof of Property 2':}
Recall that we have chosen $P^{rep}_f$ small enough so that \linebreak $P^{rep}_f \subset \pi_z(P^{rep}_{\widetilde{F}}).$
With the first assertion of the one-dimensional Property 2 and the observation $\varphi_w(z)=\Phi_w(z,x)$, we have
$$
1+\frac{\sqrt{w}}{2} D^{rep}_w \subset \varphi_w(V_w\cap P^{rep}_f) \subset \Phi_w((V_w\times\mathbb{C})\cap P^{rep}_{\widetilde{F}}).
$$
This shows that we can make $\Phi_w^{-1}(Z)$ a well-defined function on the set $1+\frac{\sqrt{w}}{2} D^{rep}_w$ in such a way that $\pi_z \circ \Phi_w^{-1}(Z) = \varphi_w^{-1}(Z)$ and $\Phi_w^{-1}(Z)=(z,x)\in P^{rep}_{\widetilde{F}}$ as discussed earlier.
It follows that
$$
\Phi_w^{-1}(1+\frac{\sqrt{w}}{2} D^{rep}_w)\subset(V_w\times\mathbb{C})\cap P^{rep}_{\widetilde{F}}.
$$
For $(z,x)$ in this set, $\Psi^{-1}_{\widetilde{F}}(z,x)$ is well-defined and $(z,x)$ converges to $(0,0)$ by similar reasoning as in the proof of Property 1'.
Hence by \eqref{RepFatCoordProp} and the corresponding one-dimensional result we get
$$
\Psi^{-1}_{\widetilde{F}}(z,x) - \psi^{-1}_f (z) = o(1), \qquad \text{as}~w\rightarrow 0
$$
and the result follows from the one-dimensional Property 2.

\hfill $\square$

\medskip
\noindent {\bf Proof of Property 3':}

The property $\mathcal{R}_{w_m}\subset\Phi_{w_m}(V_{w_m}\times\mathbb{C})$ follows directly from the one-dimensional equivalent, since $\Phi_w(z,x)=\varphi_w(z)$.

By the one-dimensional Property 3 we have
$$
\sup_{z\in\varphi_{w_m}^{-1}(\mathcal{R}_{w_m})} \left| \varphi_{g(w_m)}\circ f_{w_m} (z) - \varphi_{w_m}(z) - \frac{\sqrt{w_m}}{2}\right| = o(w_m).
$$
Note that $\Phi_{w_m}(z,x)=\varphi_{w_m}(z)$ and hence
\begin{align}\label{samesupset}
\varphi_{w_m}^{-1}(\mathcal{R}_{w_m})= \pi_z(\Phi_{w_m}^{-1}(\mathcal{R}_{w_m})) = \pi_z(\Phi_{w_m}^{-1}(\mathcal{R}_{w_m})\cap\Omega_{w_m}).
\end{align}
From here on within this proof, a supremum is always taken over the set \linebreak $\Phi_{w_m}^{-1}(\mathcal{R}_{w_m})\cap\Omega_{w_m}$, unless noted differently.
It follows from \eqref{samesupset} that if we can estimate the difference
\begin{align*}
&\sup   | \Phi_{w_{m+1}}\circ \widetilde{F}_{w_m} (z,x) - \varphi_{{g(w_m)}} \circ f_{w_m}(z) |
\end{align*}
by terms of order $o(w_m)$, then we obtain the required estimate
\begin{align}\label{ApproxTranslProp}
\sup  \left| \Phi_{w_{m+1}}\circ \widetilde{F}_{w_m} (z,x) - \Phi_{w_m}(z,x) - \frac{\sqrt{w_m}}{2}\right| = o(w_m).
\end{align}
Note that $w_{m+1}=\pi_w\circ G(w_m, y_m)$ is generally not equal to $g(w_m)$. Using
$$
\Phi_{w_{m+1}}\circ\widetilde{F}_{w_m}(z,x)=\varphi_{w_{m+1}}(\pi_z\circ\widetilde{F}_{w_m}(z,x))
$$
we estimate
\begin{align}\nonumber
&\sup   | \Phi_{w_{m+1}}\circ \widetilde{F}_{w_m} (z,x) - \varphi_{g(w_m)} \circ f_{w_m} |\\
\leq &  \sup  \left|\frac{\partial\varphi_w}{\partial z}(z)\right||\pi_z\circ \widetilde{F}_{w_m}(z,x) - f_{w_m}(z)| + \sup  \left|\frac{\partial\varphi_w}{\partial w}(z)\right| |w_{m+1} - g(w_m)|, \label{SupEstimate}
\end{align}
where in the last line the supremum is also to be taken over $w$ on the interval that joins $w_{m+1}$ and $g(w_m)$.
We will prove that each of the two terms in the last line is of order $o(w_m)$.

Recall that $\varphi_w(z) = \chi_w(  \psi_w^{-1}(z))$.
By the chain rule
$$
\frac{\partial\varphi_w}{\partial z}(z) = \frac{\partial\chi_w}{\partial \mathcal{Z}}(\psi_w^{-1}(z))\frac{\partial\psi_w^{-1}}{\partial z}(z)
$$
and
$$
\frac{\partial\varphi_w}{\partial w}(z)= \frac{\partial\chi_w}{\partial w}(\psi_w^{-1}(z)) + \frac{\partial\chi_w}{\partial \mathcal{Z}}(\psi_w^{-1}(z)) \frac{\partial\psi_w^{-1}}{\partial w}(z).
$$

We will now estimate the first term in \eqref{SupEstimate}. We have
$$
\sup_{\mathcal{Z}\in\mathcal{S}_w} |\frac{\partial \chi_w}{\partial \mathcal{Z}}(\mathcal{Z})| = 1+ O(|w|^{1/4}) < 2
$$
(see the proof of Lemma 2.2 in \cite{ABDPR}).

Hence it would be enough to estimate that
$$
\sup  \left|\frac{\partial\psi_w^{-1}}{\partial z}(z)\right||\pi_z\circ \widetilde{F}_{w_m}(z,x) - f_{w_m}(z)| = o(w_m).
$$

We compute
\begin{align*}
\frac{\partial\psi_w^{-1}}{\partial z}(z)
&= \frac{1}{2\pi i}\frac{z-\zeta^-(w)}{z-\zeta^+(w)}\frac{(z-\zeta^-(w)) - (z-\zeta^+(w))}{(z-\zeta^-(w))^2}\\
& = \frac{1}{2\pi i} \frac{\zeta^+(w)-\zeta^-(w)}{(z-\zeta^+(w))(z-\zeta^-(w))}.
\end{align*}

Also note that by \eqref{TildeFw} and Lemma \ref{LemmaOmega} for $(z,x)\in\Omega_{w_m}$ we have
\begin{align}\label{Ffdistance}
|\pi_z\circ \widetilde{F}_{w_m}(z,x)-f_{w_m}(z)| = O(z^4, {w_m}x, w_m z^2) = O(z^4, {w_m}z^2, {w_m}^3).
\end{align}
We distinguish two cases.

\noindent {\bf Case 1:} $|z|>C|\sqrt{{w_m}}|$ for a large constant $C>0$.

Since $\zeta^{\pm}(w)\sim\pm \frac{\pi i}{2}\sqrt{w}$ we obtain
\begin{align}\label{EstimateCase1}
\left|\frac{\partial\psi_w^{-1}}{\partial z}(z)\right | = O(\sqrt{{w_m}}/z^2).
\end{align}
Hence
$$
\sup   \left|\frac{\partial\psi_w^{-1}}{\partial z}(z) \right|  |\pi_z\circ \widetilde{F}_{w_m}(z,x)-f_{w_m}(z)|= O(\sqrt{{w_m}}z^2, {w_m}^{3/2}, w_m^{7/2}/z^2) =o({w_m}),
$$
using that $z=o(|{w_m}|^{1/4})$ on $ \Phi_{{w_m}}^{-1}(\mathcal{R}_{{w_m}})$ for the first term (this follows from the last line in the proof of Property 3, Step 1 \cite[p.282]{ABDPR}) and
$z>C|\sqrt{w_m}|$ for the third term.

\noindent {\bf Case 2:} $|z|<C\sqrt{{w_m}}$.

Here we need to estimate the terms $(z-\zeta^\pm(w))$ appearing in the denominator of $\frac{\partial\psi_w^{-1}}{\partial z}$ for $z\in \varphi_w^{-1}(\mathcal{R}_w)$. We define $\mathcal{R}^\prime_w$ by
$$
\mathcal{R}^\prime_w:=\left\{\mathcal{Z}\in\mathbb{C}: \frac{r_w}{20}<\Re(\mathcal Z)<1-\frac{r_w}{20}~\text{and}~-1<\Im(\mathcal Z)<1\right\}.
$$
Moreover, set
$$
\mathcal{Q}_w:=\chi^{-1}_w(\mathcal{R}_w).
$$
It is also shown in \cite[p.282]{ABDPR} (the same proof as refered to above) that
\begin{align}\label{QSubsetR}
\mathcal{Q}_w\subset\mathcal{R}^\prime_w,
\end{align}
due to the fact that $\chi_w(\mathcal{Z})=\mathcal{Z}+o(r_w)$ is close enough to the identity.
The parts of the boundary of $\psi_w(\mathcal{R}^\prime_w)$ that are closest to $\zeta^\pm(w)$ are the curves
$$
\psi_w(\{z\in\mathbb{C}: \frac{r_w}{20}<\Re(\mathcal Z)<1-\frac{r_w}{20}~\text{and}~ \Im(\mathcal{Z})= \pm 1\}).
$$
Put $\mathcal{Z}=x\pm i$ with $\frac{r_w}{20}<x<1-\frac{r_w}{20}$. Then
\begin{align*}
\psi_w(x\pm i) &= \frac{\zeta^- (w) \cdot e^{2\pi i x} e^{\mp 2 \pi} - \zeta^+ (w)}{e^{2\pi ix} e^{\mp 2\pi} - 1} \\
& = \mp\frac{e^{2\pi i x} e^{\mp 2 \pi} +1}{e^{2\pi i x} e^{\mp 2 \pi} - 1}\zeta^\pm(w) + o(w)\\
\end{align*}
using that $\zeta^+(w)=-\zeta^-(w) +o(w)$.
We see that the fractional expression is bounded away from $1$. Thus
$$
|z-\zeta^\pm(w)| > b  \zeta^\pm(w) + o(w) > c \sqrt w
$$
on the boundary curves and hence on all of $\psi_w(\mathcal{R}^\prime_w)$ for some $b, c>0$ and $w$ small enough.
Since $\varphi^{-1}_w(\mathcal{R}_w)= \psi_w\circ\chi_w^{-1}(\mathcal{R}_w)\subset \psi_w(\mathcal{R}_w^\prime)$, this shows that
\begin{align}\label{EstimateCase2}
\sup  \left|\frac{\partial\psi_w^{-1}}{\partial z} (z) \right| = O(1/\sqrt{w})
\end{align}
and using $z=O(\sqrt{w_m})$ we obtain
$$
\sup   |\frac{\partial\psi_w^{-1}}{\partial z}(z) | |\pi_z\circ \widetilde{F}_{w_m}(z,x)-f_{w_m}(z)|  = O(z^4/\sqrt{w_m},\sqrt{w_m} z^2, w_m^{5/2})=o(w_m).
$$
This finishes the proof that the first term in equation \eqref{SupEstimate} is of order $o(w_m)$.

For the second term we compute
$$
\frac{\partial\psi_w^{-1}}{\partial w}(z) = \frac{1}{2\pi i}\frac{ (\zeta^-)^\prime(w) (z-\zeta^+(w))- (\zeta^+)^\prime(w) (z-\zeta^-(w))}{(z-\zeta^+(w))(z-\zeta^-(w))},
$$
with $(\zeta^\pm)^\prime(w) = O (1/\sqrt{w})$.
We distinguish the same cases and use the same estimates on the terms $(z-\zeta^\pm(w))$ as before to obtain
$$\sup  \left| \frac{\partial\psi_w^{-1}}{\partial w}(z) \right| = O\left(\frac{1}{w}\right).$$

Now
\begin{align*}
\frac{\partial \chi_w}{\partial w} (\mathcal{Z}) &= -\frac{(1-a)}{4\sqrt{w}} \log\left(\frac{2\sin(\pi\mathcal{Z})}{\pi\sqrt{w}}\right)+\frac{\sqrt{w}(1-a)}{2}\frac{\pi\sqrt{w}}{2 \sin(\pi \mathcal{Z})}\frac{2\sin(\pi\mathcal{Z})}{2\pi w^{3/2}}\\
&= O\left(\frac{\log(w)}{\sqrt{w}}\right)+O\left(\frac{1}{\sqrt{w}}\right) = O\left(\frac{1}{w}\right).
\end{align*}

Recall also that $$\sup_{\mathcal{Z}\in\mathcal{S}_w} |\frac{\partial \chi_w}{\partial \mathcal{Z}}(\mathcal{Z})| = O(1).$$

In addition
\begin{align*}
|w_{m+1} - g(w_m)|  &= |\pi_w\circ G(w_m, y_m) - g(w_m)| \\
& = O(w_m^3, w_my_m, y_m^2)\\
&= O(w_m^3)
\end{align*}
by Lemma \ref{LemmaOmegaG}.
Hence
\begin{align*}
&\sup  \left|\frac{\partial\varphi_w}{\partial w}(z)\right| |w_{m+1} - g(w_m)| \\
=& \sup  \left|\frac{\partial\chi_w}{\partial w}(\psi_w^{-1}(z)) + \frac{\partial\chi_w}{\partial \mathcal{Z}}(\psi_w^{-1}(z)) \frac{\partial\psi_w^{-1}}{\partial w}(z)\right|  |w_{m+1} - g(w_m)|  \\
=& o(w_m)
\end{align*}
and we have finished the proof of \eqref{ApproxTranslProp}.

It is left to show that
$$
\widetilde{F}_{w_m}(\Phi_{w_m}^{-1}(\mathcal{R}_{w_m}) \cap \Omega_{w_m})\subset (V_{w_{m+1}}\times\mathbb{C})\cap \Omega_{w_{m+1}}.
$$
By Lemma \ref{LemmaOmega} we have $\widetilde{F}_{w_m}(\Omega_{w_m})\subset \Omega_{w_{m+1}}$.
So it is enough to show that \linebreak $\pi_z\circ \widetilde{F}_{w_m}(\Phi_{w_m}^{-1}(\mathcal{R}_{w_m})\cap\Omega_{w_m} )\subset V_{w_{m+1}}$.
In the proof of the one-dimensional Property 3 in \cite{ABDPR} it is shown that
$$
f_w\circ\varphi_w^{-1}(\mathcal{R}_w)\subset V_{g(w)}.
$$
In fact, it turns out that there is a significant margin between these two sets. Our goal is to show that by replacing the one-dimensional maps $f_w\circ\varphi_w^{-1}$ by $\widetilde{F}_w\circ\Phi_w^{-1}$ we introduce an error that is strictly smaller than this margin.

\subsubsection*{Estimate of the margin}

From \cite[p.286]{ABDPR} we have
$$
f_{w_m}\circ\varphi_{w_m}^{-1}(\mathcal{R}_{w_m})=\psi_{g(w_m)}(\mathcal{F}^1\circ\mathcal{F}^0(\mathcal{Q}_{w_m})),
$$
where $\mathcal{F}^1\circ\mathcal{F}^0(\mathcal{Z})=\mathcal{Z}+O(w^{1/2})$ for $\mathcal{Z}\in\mathcal{Q}_w$ as $w\rightarrow 0$.
From \eqref{QSubsetR} we have that $\mathcal{Q}_w\subset\mathcal{R}^\prime_w$. If necessary, we can replace $r_w/20$ by $r_w/30$ in the definition of $\mathcal{R}^\prime_w$ to ensure that the minimal distance between the boundaries of $\mathcal{Q}_w$ and $\mathcal{R}_w^\prime$ is comparable to $r_w$. Since $|w|^{1/2}=o(r_w)$ we obtain $\mathcal{F}^1\circ\mathcal{F}^0(\mathcal{Q}_{w_m})\subset \mathcal{R}^\prime_{w_m}\subset \mathcal{R}^\prime_{w_{m+1}},$
hence
$$
f_{w_m}\circ\varphi_{w_m}^{-1}(\mathcal{R}_{w_m})\subset \psi_{g(w_m)}(\mathcal{R}_{w_{m+1}}^\prime).
$$
In addition we know that the minimal distance between the boundaries of $\mathcal{R}_w^\prime$ and $\mathcal{S}_w$ is comparable to $|w|^{1/4}$.
By our earlier estimates \eqref{EstimateCase1} and \eqref{EstimateCase2} we have \linebreak
$|\frac{\partial\psi_w^{-1}}{\partial z}| = O(w^{-1/2}).$
It follows that the minimal distance between
$\psi_{g(w_m)}(\mathcal{R}_w^\prime)$ and \linebreak $\psi_{g(w_m)}(\mathcal{S}_{w_{m+1}})$
at least of order $|w_m|^{3/4}$.
Replacing $g(w_m)$ by $w_{m+1}$ does not change this order, because
$|\frac{\partial\psi}{\partial w}|=O(1/|w|^{1/2})$ and $|g(w_m)-w_{m+1}|= O(w_m^3)$.
Therefore we finally obtain that the size of the margin between $f_{w_m}\circ\varphi_{w_m}^{-1}(\mathcal{R}_{w_m})$ and $\psi_{w_{m+1}}(S_{w_{m+1}})=V_{w_{m+1}}$ is comparable to $|w_m|^{3/4}$.

\subsubsection*{Difference between dimension 1 and 2}

Recall from \eqref{Ffdistance} that for $(z,x)\in\Phi_{w_m}^{-1}(\mathcal{R}_{w_m}) \cap \Omega_{w_m}$:
$$|\pi_z\circ \widetilde{F}_{w_m}(z,x)-f_{w_m}(z)| = O(z^4, {w_m}z^2, {w_m}^3) = o (w_m),$$
since $z=o(|w|^{1/4})$ on $\varphi_w^{-1}(\mathcal{R}_w)=\pi_z\Phi_w^{-1}(\mathcal{R}_w)$.

We conclude that the difference between the sets $f_{w_m}(\varphi_{w_m}^{-1}(\mathcal{R}_{w_m}))$ and \linebreak $\pi_z\circ\widetilde{F}_{w_m}( \Phi_{w_m}^{-1}(\mathcal{R}_{w_m})\cap \Omega_{w_m})$ is of order $o(w_m)$, which is smaller than the margin as shown above, completing the proof.

\hfill $\square$

\section{Convergence to the Lavaurs map and proof of Proposition A'}

\subsection{Notation}

Let $C_F$ be a compact subset of $\mathcal{B}_F$ and $C_G$ be a compact subset of $\mathcal{B}_G$. Our goal in this section is to prove that $H^{2n+1}(z,x,G^{n^2}(w,y))$ converges uniformly on $C_F\times C_G$ to the map $(z,x,w,y)\mapsto(\mathcal{L}_F(z,x),0,0)$.
For $(z,x,w,y)\in C_F\times C_G$ we set
$$w_m:=\pi_w(G^m(w,y))$$
and for $m_2\geq m_1\geq 0 $ we set
$$\boldsymbol{F}_{m_2,m_1} := F_{w_{m_2 -1}}\circ\dots\circ F_{w_{m_1}} \qquad \text{with} \qquad F_w(z,x)= F(z,x)+(\frac{\pi^2}{4} w,0).$$
We saw in Lemma \ref{LemmaFwlocal} how $F_w$ translates into a locally defined map $\widetilde{F}_w := U \circ F_w \circ(U)^{-1}$ with
$$
\widetilde{F}_w:=U \circ F_w \circ U^{-1}\begin{pmatrix}z\\x \end{pmatrix} = \widetilde{F}\begin{pmatrix}z\\x \end{pmatrix} + \begin{pmatrix}\frac{\pi^2}{4}w+ \mathcal{O}(wx, wz^2)\\0 \end{pmatrix}.
$$
We define
$$\widetilde{\boldsymbol{F}}_{m_2,m_1} := \widetilde{F}_{w_{m_2 -1}}\circ\dots\circ \widetilde{F}_{w_{m_1}},$$
accordingly.

Let $\Phi_F$ and $\Psi_F$ be the attracting and repelling Fatou coordinates for the map $F$, respectively. They were constructed for the local map $\widetilde{F}=U\circ F\circ U^{-1}$, resulting in local Fatou coordinates $\Phi_{\widetilde{F}}$ and $\Psi_{\widetilde{F}}$. As for the local coordinate change $U$, these maps are only defined in a neighborhood $V$ of the origin. Recall however that we may assume that the attracting and repelling petals are chosen small enough to be contained in $V$. The global Fatou coordinates are defined via $\Phi_F := \Phi_{\widetilde{F}}\circ U$ and $\Psi_F:=U^{-1}\circ\Psi_{\widetilde{F}}$ and are extended to $\mathcal{B}_F$ and $\mathbb{C}$, respectively, by using the functional equations $\Phi_F\circ F = T_1\circ \Phi_F$ and $F\circ \Psi_F = \Psi_F \circ T_1$.

The notation $o(\cdot)$ or $O(\cdot)$ stands for estimates that are uniform on $C_F\times C_G$ and with respect to $k\in[0,2n+1]$, and depend only on $n$. We set 
$$
k_n:=\lfloor n^\alpha\rfloor, \qquad \text{where}\qquad \frac{1}{2}<\alpha<\frac{2}{3}.
$$

\subsection{Proof of Proposition A'}

We will first prove Proposition A' under the assumption that the following analogues of Propositions 3.1. through 3.4. in \cite{ABDPR} hold.

\begin{prop}\label{propos1}
Let $(z,x)\in C_F$. There exists $\kappa_0\geq 1$ such that the following properties hold.
\begin{enumerate}[(i)]
\item The point $(z^\iota_n, x^\iota_n),$ given by
\begin{align}\label{defzin}
(z^\iota_n, x^\iota_n):=\widetilde{\boldsymbol{F}}_{n^2+k_n, n^2+\kappa_0}\circ U \circ \boldsymbol{F}_{n^2+\kappa_0,n^2} (z,x)
\end{align}
 is well-defined for $n$ large enough.
\item We have $z^\iota_n\sim -1/k_n$ and $(z^\iota_n, x^\iota_n)\in\Omega_{w_{n^2+k_n}}\cap \Delta^2_s$ as defined in Lemma \ref{LemmaOmega}.
Moreover,  $$\Phi_{\widetilde{F}}(z^\iota_n,x^\iota_n) = \Phi_F(z,x)+k_n +o(1),$$ as $n\rightarrow \infty$.
\end{enumerate}
\end{prop}

Proposition \ref{propos1} concerns entering the eggbeater and is very similar to the one-dimensional case, as is Proposition \ref{propos2}.

\begin{prop}\label{propos2}
As $n\rightarrow\infty$,
$$2n\left(\sum_{m=n^2+k_n}^{n^2+2n-k_n}\frac{\sqrt{w_m}}{2} \right) = 2n - 2k_n + o(1).$$
\end{prop}

\begin{prop}\label{propos3}
Let $(z^\iota_n, x^\iota_n)$ be a sequence such that $z^\iota_n\sim -1/k_n$ and \linebreak $(z^\iota_n, x^\iota_n)\in P^{att}_{\widetilde{F}}\cap\Omega_{w_{n^2+k_n}}\cap \Delta^2_s$ for $n$ large enough.
Set
\begin{align}\label{defzon}
(z^o_n,x^o_n):=\widetilde{\boldsymbol{F}}_{(n+1)^2-k_n, n^2+k_n}(z^\iota_n,x^\iota_n).
\end{align}
Then $z^o_n \sim 1/k_n$ and there exists $u^o_n\in\mathbb{C}$
such that  $(z^o_n,u^o_n)\in P^{rep}_{\widetilde{F}}$
for $n$ large enough.
Moreover,
$$\Psi_{\widetilde{F}}^{-1}(z^o_n,u^o_n) 
= \Phi_{\widetilde{F}}(z^\iota_n,x^\iota_n) - 2k_n +o(1)$$
as $n\rightarrow\infty$.
\end{prop}

The proof of Proposition \ref{propos3} (passing trough the eggbeater) is based on the properties 1', 2' and 3' of the approximate Fatou coordinates, introduced in the previous section.

Finally, Proposition \ref{propos4} describes the situation when leaving the eggbeater:

\begin{prop}\label{propos4}
Let $z^o_n, x^o_n$ and $u^o_n$ be the sequences as in Proposition \ref{propos3}. Then there exists $\kappa_1\geq 1$ for which
$$
\widetilde{\boldsymbol{F}}_{(n+1)^2-\kappa_1, (n+1)^2-k_n}(z^o_n,x^o_n) = \widetilde{F}^{k_n-\kappa_1}(z^o_n,u^o_n)+o(1),
$$
and the coordinate change $U^{-1}$ is defined at the point on the left hand side.
\end{prop}

After passing though the eggbeater the iterates will in general not lie inside the repelling petal $P^{rep}_{\widetilde{F}}$, since this is only one-dimensional. Therefore, the proof of Proposition \ref{propos4} will be more complicated than in the one-dimensional case, because we have to show that the errors made by repeatedly ``jumping back into the repelling petal'' are small enough.

Let us now prove Proposition A' based on Propositions \ref{propos1} through \ref{propos4}.

\medskip
\noindent {\bf Proof of Proposition A' :} Let $(z^\iota_n,x^\iota_n)$ and $(z^o_n,x^o_n)$, as well as $u^o_n$ be defined as in Propositions \ref{propos1} and \ref{propos3}. Note that since $G^m(w,y)\rightarrow(0,0)$ for $(w,y)\in C_G\subset\mathcal{B}_G$ as $m\rightarrow\infty$, Proposition A' is equivalent to the statement $\boldsymbol{F}_{(n+1)^2,n^2}(z,x)=\mathcal{L}_F(z,x)+o(1)$ for $(z,x)\in C_F\subset\mathcal{B}_F$ as $n\rightarrow\infty$.
From Proposition \ref{propos1} we obtain that \linebreak $\Phi_{\widetilde{F}}(z^\iota_n,x^\iota_n)=\Phi_F(z,x)+k_n+o(1)$. Combined with the result of Proposition \ref{propos3} this gives
\begin{align}\label{eqnPsiPhi}
\Psi_{\widetilde{F}}^{-1}(z^o_n,u^o_n)=\Phi_F(z,x)-k_n+o(1).
\end{align}
Since the sequence $(F_{w_{(n+1)^2-k}})_{n\geq 0}$ converges to $F$ locally uniformly for $k\in[0,\kappa_1]$ as $n\rightarrow\infty$, we have
$$
\boldsymbol{F}_{(n+1)^2,(n+1)^2-\kappa_1} = F^{\kappa_1} + o(1).
$$
It follows that
$$
\boldsymbol{F}_{(n+1)^2,n^2}(z,x) =  F^{\kappa_1}\circ \boldsymbol{F}_{(n+1)^2-\kappa_1, n^2+\kappa_0} \circ \boldsymbol{F}_{n^2+\kappa_0,n^2}(z,x)+o(1).
$$
Introducing the coordinate change $U$ and using the definitions of $(z^\iota_n, x^\iota_n)$ and $(z^o_n, x^o_n)$ in \eqref{defzin} and \eqref{defzon}, we obtain
\begin{align*}
\boldsymbol{F}_{(n+1)^2,n^2}(z,x) &=  F^{\kappa_1}\circ U^{-1} \circ \widetilde{\boldsymbol{F}}_{(n+1)^2-\kappa_1, n^2+\kappa_0} \circ U \circ \boldsymbol{F}_{n^2+\kappa_0,n^2}(z,x)+o(1)\\
&=F^{\kappa_1}\circ U^{-1} \circ \widetilde{\boldsymbol{F}}_{(n+1)^2-\kappa_1, (n+1)^2-k_n}(z^o_n,x^o_n)+o(1).
\end{align*}
Proposition \ref{propos4} and equation \eqref{eqnPsiPhi} give
\begin{align*}
\boldsymbol{F}_{(n+1)^2,n^2}(z,x) &= F^{\kappa_1}\circ U^{-1} \circ \widetilde{F}^{k_n-\kappa_1}(z^o_n,u^o_n) +o(1)\\
&=F^{\kappa_1}\circ U^{-1} \circ \widetilde{F}^{k_n-\kappa_1}\circ \Psi_{\widetilde{F}}\circ\Psi_{\widetilde{F}}^{-1}(z^o_n,u^o_n) +o(1)\\
&=F^{\kappa_1}\circ U^{-1} \circ \widetilde{F}^{k_n-\kappa_1}\circ \Psi_{\widetilde{F}}(\Phi_F(z,x)-k_n) +o(1)\\
&=F^{\kappa_1}\circ U^{-1} \circ \Psi_{\widetilde{F}}(\Phi_F(z,x)-k_n+k_n-\kappa_1) +o(1)\\
&=F^{\kappa_1}\circ \Psi_{F}(\Phi_F(z,x)-\kappa_1) +o(1)\\
&=\Psi_{F}(\Phi_F(z,x)) +o(1)\\
&=\mathcal{L}_F(z,x)+o(1).
\end{align*} \hfill $\square$

\subsection{Proof of Proposition \ref{propos2}}

Let $\Phi_{\widetilde{G}}:P^{att}_{\widetilde{G}}\rightarrow \mathbb{C}$ denote the attracting Fatou coordinate of $\widetilde{G}=U\circ G\circ U^{-1}$ and $\Phi_G=\Phi_{\widetilde{G}}\circ U^{-1}$ the attracting Fatou coordinate of $G$.
Similar to property \eqref{AttrFatCoordProp}, as $n\rightarrow \infty$, we have
$$
\Phi_{\widetilde{G}}(w_{n^2+k},y_{n^2+k}) = \frac{1}{w_{n^2+k}} + c \log\left(\frac{1}{w_{n^2+k}}\right) + o(1),
$$
for a constant $c$. Note that compared to \eqref{AttrFatCoordProp} the different signs arise due the fact that $g(w)=w-w^2+\dots$  as opposed to $f(z)=z+z^2+\dots$.
Since $U$ has linear part equal to the identity, $U^{-1}$ maps $w$ to $w+O(w^2,wy,y^2)$. The latter is equal to $w+O(w^2)$ for $w$ small enough by Lemma \ref{LemmaOmegaG}.
Therefore
\begin{align*}
\Phi_G(w_{n^2+k},y_{n^2+k}) &= \frac{1}{w_{n^2+k}+O(w_{n^2+k}^2)} + c \log\left(\frac{1}{w_{n^2+k}+O(w_{n^2+k}^2)}\right) + o(1)\\
& = \frac{1}{w_{n^2+k}} + c \log\left(\frac{1}{w_{n^2+k}}\right) + O(1).
\end{align*}
We also have
$$\Phi_G(w_{n^2+k},y_{n^2+k}) = \Phi_G(w,y) + n^2+k = n^2+k + O(1).$$
It follows that
\begin{align}\label{wn}
w_{n^2+k} = \frac{1}{n^2+k+O(\log n)},
\end{align}
for $k\in[k_n, 2n-k_n]$ and from there the result follows exactly as in the proof of Proposition 3.2. in \cite{ABDPR}.
\hfill $\square$

\subsection{Proof of Proposition \ref{propos1}}

This proof is to a large extent similar to the proof of Proposition 3.1 in \cite{ABDPR}.
\subsubsection*{Step 1}
We choose $\kappa_0\geq 1$ large enough so that
$$F^{\kappa_0}(C_F)\subset P^{att}_F.$$
In addition to $(z^\iota_n, x^\iota_n)=\widetilde{\boldsymbol{F}}_{n^2+k_n, n^2+\kappa_0}\circ U \circ \boldsymbol{F}_{n^2+\kappa_0,n^2} (z,x),$ we set
\begin{align}\label{defz0x0}
(z_0,x_0):=U\circ\boldsymbol{F}_{n^2+\kappa_0, n^2}(z,x),
\end{align}
for $(z,x)\in C_F$.
Since for every fixed $k>0$ the sequence of polynomials $(F_{w_{n^2+k}})_{n\geq 0}$ converges to $F$ locally uniformly, we get that for $k\in[1,\kappa_0]$, the sequence $\boldsymbol{F}_{n^2+k, n^2}$ converges uniformly to $F^k$ on $C_F$. If $n$ is large enough, then
$$\boldsymbol{F}_{n^2+k, n^2}(C_F)\subset \mathcal{B}_F \text{~ for ~} k\in[1,\kappa_0], \text{~ and ~} \boldsymbol{F}_{n^2+\kappa_0, n^2}(C_F)\subset P^{att}_F.$$
Recall the definition $P^{att}_F=U^{-1}(P^{att}_{\widetilde{F}})$. It follows that $U\circ\boldsymbol{F}_{n^2+\kappa_0, n^2}(C_F) \subset P^{att}_{\widetilde{F}}$.

Since $\boldsymbol{F}_{n^2+\kappa_0, n^2} $ converges to $F^{\kappa_0}$ as $n\rightarrow\infty$, we have
\begin{align*}
\Phi_{\widetilde{F}}(z_0,x_0) &= \Phi_{\widetilde{F}} \circ U\circ\boldsymbol{F}_{n^2+\kappa_0, n^2}(z,x)\\
  &= \Phi_F \circ \boldsymbol{F}_{n^2+\kappa_0, n^2} (z,x) \\
  &= \Phi_F \circ F^{\kappa_0} (z,x) +o(1)\\
  &= \Phi_F(z,x) + \kappa_0 +o(1)
\end{align*}
and in addition for $n$ large enough
\begin{align}\label{kn10}
k_n>\frac{10}{|z_0|}, \text{~for~} (z,x)\in C_F,
\end{align}
which we will use later on.

\subsubsection*{Step 2}

We show that for $n$ large enough and $k\in[\kappa_0, k_n]$, $\widetilde{\boldsymbol{F}}_{n^2+k, n^2+\kappa_0}\circ U \circ \boldsymbol{F}_{n^2+\kappa_0, n^2}(C_F)\subset P^{att}_{\widetilde{F}}$. In this case, $(z^\iota_n, x^\iota_n)$ is well-defined. We already have $(z_0,x_0)=U \circ \boldsymbol{F}_{n^2+\kappa_0, n^2}(z,x)\in P^{att}_{\widetilde{F}}$ for $(z,x)\in C_F$ by Step 1. Hence it is enough to show that $\widetilde{\boldsymbol{F}}_{n^2+k, n^2+\kappa_0}(z_0,x_0)\in P^{att}_{\widetilde{F}}$.
Recall that the attracting petal is defined by
$$
P^{att}_{\widetilde{F}}:=\{(z,x)\in\mathbb{C}^2: \Re\left(-\frac{1}{z}\right)>R, |x|<\eta\}
$$
and that $x^\prime:=\pi_x\circ \widetilde{F}_w(z,x) = \delta x + O(x^2, z)$. For $\eta$ small enough we see that \linebreak $|x^\prime| < \beta\eta + M/R$ for some $\beta$, $\delta<\beta<1$ and a constant $M>0$. Hence $|x^\prime|<\eta$, if $R$ is large enough. 

To see that also the $z$-coordinate stays in the set $\left\{z\in\mathbb{C}:\Re\left(-\frac{1}{z}\right)>R\right\}$, we work in the coordinate $Z=-1/z, X=x$.
Consider the map
$$
\mathcal{F}(Z, X)= -\frac{1}{\pi_z\circ \widetilde{F}(-1/Z, X)} = Z + 1 + O(1/Z)
$$
and its perturbations
$$
\mathcal{F}_m(Z,X) :=  -\frac{1}{\pi_z \circ \widetilde{F}_{w_m}(-1/Z, X)}.
$$
Using that 
$$
\pi_z\circ \widetilde{F}_w(z,x)= \pi_z\circ\widetilde{F}(z,x)+\frac{\pi^2}{4}w+O(wx,wz^2)=\pi_z\circ\widetilde{F}(z,x)+O(w)
$$
on $P^{att}_{\widetilde{F}}$ by Lemma \ref{LemmaFwlocal}, we obtain
$$
\mathcal{F}_m(Z,X)  = \mathcal{F}(Z,X) + \frac{[\mathcal{F}(Z,X)]^2\cdot O(w_m)}{1 + \mathcal{F}(Z,X)\cdot O(w_m)}.
$$
We are interested in $m=n^2+k$ with $k\in[\kappa_0,k_n]$.
Note that since $\mathcal{F}(Z,X)=O(Z)$ as $|Z|\rightarrow \infty$ and $w_{n^2+k}=O(1/n^2)$ by \eqref{wn} we get for $k\in[\kappa_0,k_n]$,
$$
\sup_{|Z|=R, |X|<\eta} |\mathcal{F}_{n^2+k}(Z,X)-\mathcal{F}(Z,X)| = o(1)
$$
and
$$
\sup_{|Z|=2k_n, |X|<\eta} |\mathcal{F}_{n^2+k}(Z,X)-\mathcal{F}(Z,X)| = O\left(\frac{k_n^2}{n^2}\right)=o(1).
$$
The maximum principle implies that on the set $\{(Z,X)\in\mathbb{C}^2: R<|Z|<2k_n, |X|<\eta\}$ the remainder term $\mathcal{F}_{n^2+k}(Z,X)-\mathcal{F}(Z,X)$ with $k\in[\kappa_0,k_n]$ is negligible. In particular
$$
\sup_{R<|Z|<2k_n, |X|<\eta} |\mathcal{F}_{n^2+k}(Z,X)- Z - 1| < 1/10
$$
if $n$ is large enough.
We will show by induction that for every $k\in[\kappa_0,k_n]$ we have
\begin{align}\nonumber
-\frac{1}{\pi_z\circ \widetilde{\boldsymbol{F}}_{n^2+k,n^2+\kappa_0}(z_0,x_0)} & \in  \overline{D}\left(-\frac{1}{z_0}+k-\kappa_0, \frac{k-\kappa_0}{10}\right)\\ \label{2kn}
& \subset  \{Z\in \mathbb{C}: \Re(Z)>R, |Z|<2 k_n \}.
\end{align}
In this case it follows that $\widetilde{\boldsymbol{F}}_{n^2+k, n^2+\kappa_0}(z_0,x_0)\in P^{att}_{\widetilde{F}}$.
The induction hypothesis clearly holds for $k=\kappa_0$.
If it holds for some $k\in[\kappa_0,k_n-1]$, then
\begin{align*}
& -\frac{1}{\pi_z\circ \widetilde{\boldsymbol{F}}_{n^2+k+1,n^2+\kappa_0}(z_0,x_0)}\\
=& \mathcal{F}_{n^2+k}\left(-\frac{1}{\pi_z\circ\widetilde{\boldsymbol{F}}_{n^2+k,n^2+\kappa_0}(z_0,x_0)},\pi_x\circ\widetilde{\boldsymbol{F}}_{n^2+k,n^2+\kappa_0}(z_0,x_0) \right)\\
 \in&    \overline{D}\left(-\frac{1}{\pi_z\circ \widetilde{\boldsymbol{F}}_{n^2+k,n^2+\kappa_0}(z_0,x_0)}+1, \frac{1}{10}\right)\\
\subset&   \overline{D}\left(-\frac{1}{z_0}+k-\kappa_0+1, \frac{k-\kappa_0}{10}+\frac{1}{10}\right).
\end{align*}
If $Z$ belongs to the latter disk, then
$$\Re(Z)>\Re\left(-\frac{1}{z_0}\right)+k-\kappa_0+1-\frac{k-\kappa_0+1}{10}>R+\frac{9}{10}(k-\kappa_0+1)>R.$$
On the other hand, using \eqref{kn10}, we have
$$
|Z|<\left|\frac{1}{z_0}\right|+k-\kappa_0+1+\frac{k-\kappa_0+1}{10}<\frac{1}{10}k_n+\frac{11}{10} k_n < 2k_n,
$$
which finishes the induction and with that the proof of (i).

\subsubsection*{Step 3}

In this step we prove that
$$
\Phi_{\widetilde{F}}(\widetilde{\boldsymbol{F}}_{n^2+k_n, n^2+\kappa_0}(z_0,x_0))=\Phi_{\widetilde{F}} (z_0,x_0) +k_n -\kappa_0 + o(1).
$$
By \eqref{AttrFatCoordProp}, we have that
$$
\frac{\partial\Phi_{\widetilde{F}}}{\partial z}(z,x) = \frac{1}{z^2} + O\left(\frac{1}{z}\right),
$$
as $P^{att}_{\widetilde{F}}\ni(z,x)\rightarrow(0,0).$
In addition, by \eqref{2kn}, we have for $k\in[\kappa_0, k_n]$,
$$
\left|\frac{1}{\pi_z\circ \widetilde{\boldsymbol{F}}_{n^2+k,n^2}(z_0,x_0)}\right| \leq 2 k_n.
$$
We infer that
$$
\sup \left|\frac{\partial\Phi_{\widetilde{F}}}{\partial z}(z,x) \right| = O (k_n^2),
$$
where the supremum is taken over paths that join $\widetilde{\boldsymbol{F}}_{n^2+k,n^2+\kappa_0}(z_0,x_0)$ and \linebreak $\widetilde{\boldsymbol{F}}_{n^2+k+1,n^2+\kappa_0}(z_0,x_0)$ for $k\in[\kappa_0, k_n-1]$.
It follows that
\begin{align*}
\Phi_{\widetilde{F}} ( \widetilde{\boldsymbol{F}}_{n^2+k+1,n^2+\kappa_0}(z_0,x_0)) &=  \Phi_{\widetilde{F}} \left( \widetilde{F} \circ \widetilde{\boldsymbol{F}}_{n^2+k,n^2+\kappa_0}(z_0,x_0)  + (O(w_{n^2+k}),0) \right)\\
&=   \Phi_{\widetilde{F}} ( \widetilde{F} \circ \widetilde{\boldsymbol{F}}_{n^2+k,n^2+\kappa_0}(z_0,x_0)) + \sup  \left|\frac{\partial\Phi_{\widetilde{F}}}{\partial z}(z,x) \right| \cdot O\left(\frac{1}{n^2}\right) \\
&=  \Phi_{\widetilde{F}}( \widetilde{\boldsymbol{F}}_{n^2+k,n^2+\kappa_0}(z_0,x_0)) + 1 + O\left(\frac{k_n^2}{n^2}\right),
\end{align*}
for $k\in[\kappa_0, k_n-1].$
Thus for, $k=k_n$:
\begin{align*}
\Phi_{\widetilde{F}} ( \widetilde{\boldsymbol{F}}_{n^2+k_n,n^2+\kappa_0}(z_0,x_0)) &= \Phi_{\widetilde{F}} (z_0,x_0) +k_n-\kappa_0 + O\left(\frac{k_n^3}{n^2}\right)\\
&= \Phi_{\widetilde{F}} (z_0,x_0) +k_n-\kappa_0 + o(1),
\end{align*}
since $k_n\sim n^\alpha$ and $\alpha<2/3.$
\subsubsection*{Step 4}
We combine the results of Step 3 and Step 1 to obtain
\begin{align*}
\Phi_{\widetilde{F}}(z^\iota_n, x^\iota_n)&=\Phi_{\widetilde{F}} \circ \widetilde{\boldsymbol{F}}_{n^2+k_n, n^2+\kappa_0}(z_0,x_0)\\
& = \Phi_{\widetilde{F}} (z_0,x_0) +k_n-\kappa_0 + o(1)\\
&= \Phi_F(z,x)+k_n +o(1),
\end{align*}
which was the main assertion of (ii).
In particular we have $\Phi_{\widetilde{F}}(z^\iota_n, x^\iota_n) = k_n + O(1)$ and thus also
$$z^\iota_n \sim -1/\Phi_{\widetilde{F}}(z^\iota_n, x^\iota_n) \sim -1/k_n.$$

\subsubsection*{Step 5}

In order to complete the proof of (ii), we still need to show that $(z^\iota_n, x^\iota_n)\in\Omega_{w_{n^2+k_n}}\cap \Delta^2_s$. Let $s>0$ be so small as needed for Lemma \ref{LemmaOmega}. Let $(z,x)\in C_F\subset \mathcal{B}_F$.
If necessary we can take the attracting petals smaller and $\kappa_0$ larger, to make sure that we have $F^{\kappa_0}(C_F)\subset P^{att}_F\subset U^{-1}(\Delta^2_s)$.
Note that $\boldsymbol{F}_{n^2+\kappa_0, n^2}$ converges to $F^{\kappa_0}$ locally uniformly as $n\rightarrow \infty$, hence for $n$ large $\boldsymbol{F}_{n^2+\kappa_0, n^2}(C_F)$ is close to $F^{\kappa_0}(C_F)$. This set is bounded away from $(0,0)$ and since the linear part of the coordinate change $U$ is the identity map, applying $U$ does not change this property. Hence for $(z_0,x_0)=U\circ\boldsymbol{F}_{n^2+\kappa_0, n^2}(z,x)$, there exists a constant $C^\prime$ such that $|x_0|<C^\prime|z_0|^2$.
In the definition of
$$
\Omega_{w_n}:=\{(z,x)\in\mathbb{C}^2: |x|<C\max\{|z|^2, |w_n|^2\}\}
$$
we can choose $C$ to be at least as large as $C^\prime$.
Then $(z_0,x_0)\in\Omega_{w_{n^2+\kappa_0}}\cap\Delta^2_s$. Moreover by Step 2, the iterates $\widetilde{F}^j(z_0,x_0)$ stay in $P^{att}_{\widetilde{F}}\subset\Delta^2_s$ for $j\in[0,k_n-\kappa_0]$ and applying Lemma \ref{LemmaOmega} multiple times we obtain $(z^\iota_n, x^\iota_n)\in\Omega_{w_{n^2+k_n}}\cap \Delta^2_s$, as desired.

\hfill $\square$

 \subsection{Proof of Proposition \ref{propos3}}

We use the notation $v^\iota_n:=w_{n^2+k_n}$ as in \cite{ABDPR}.
For $n$ large enough we have $(z^\iota_n, x^\iota_n)\in P^{att}_{\widetilde{F}}$. Moreover,
$$\Phi_{\widetilde{F}}(z^\iota_n, x^\iota_n) \sim -\frac{1}{z^\iota_n}\sim k_n\sim n^\alpha \sim |v^\iota_n|^{-\alpha/2},$$
hence $(z^\iota_n,x^\iota_n)\in \Phi_{\widetilde{F}}^{-1}(D^{att}_{v^\iota_n})$. Also, by Proposition \ref{propos1}  $(z^\iota_n,x^\iota_n)\in\Omega_{v^\iota_n}$. By Property 1' it follows that $(z^\iota_n,x^\iota_n)\in(V_{v^\iota_n}\times\mathbb{C})\cap\Omega_{v^\iota_n}\cap P^{att}_{\widetilde{F}}$ and
\begin{align}\label{propos3step1}
\frac{2}{\sqrt{v^\iota_n}}\Phi_{v^\iota_n}(z^\iota_n, x^\iota_n) =  \Phi_{\widetilde{F}}(z^\iota_n, x^\iota_n) + o(1).
\end{align}

We will prove by induction on $m\in[n^2+k_n, (n+1)^2-k_n]$ that
\begin{align}\label{Propos3Step21}
(z_m, x_m):=\widetilde{\boldsymbol{F}}_{m, n^2+k_n}(z^\iota_n, x^\iota_n)\in (V_{w_m}\times\mathbb{C})\cap\Omega_{w_m}
\end{align}
and
\begin{align}\label{Propos3Step22}
\Phi_{w_m}(z_m, x_m) = \Phi_{v^\iota_n}(z^\iota_n, x^\iota_n)+\sum_{j=n^2+k_n}^{m-1}\left(\frac{\sqrt{w_j}}{2}+o\left(\frac{1}{n^2}\right)\right).
\end{align}
For $m=n^2+k_n$, we have $w_m=v^\iota_n$ and hence $(z^\iota_n, x^\iota_n)\in (V_{v^\iota_n}\times\mathbb{C})\cap\Omega_{v^\iota_n}$ by our previous discussion.
Suppose the induction hypothesis holds for some $m$ in the interval $[n^2+k_n, (n+1)^2-k_n-1]$.
According to \eqref{wn} we have
$$
\sqrt{w_m} = \frac{1}{\sqrt{n^2+O(n))}} = \frac{1}{n}+O\left(\frac{1}{n^2}\right).
$$
In addition, by \eqref{propos3step1},
\begin{align}\label{Propos3Step23}
\Phi_{v^\iota_n}(z^\iota_n, x^\iota_n) = \frac{\sqrt{v^\iota_n}}{2}(\Phi_{\widetilde{F}}(z^\iota_n, x^\iota_n) + o(1)) = \frac{k_n}{2n}+o\left(\frac{k_n}{n}\right).
\end{align}

Combining with the induction assumption, we deduce that
$$
\Phi_{w_m}(z_m, x_m) = \frac{k_n}{2n}+o\left( \frac{k_n}{n}\right) + (m-n^2-k_n)\left(\frac{1}{2n}+O\left(\frac{1}{n^2}\right)\right) = \frac{m-n^2}{2n}+o\left( \frac{k_n}{n}\right).
$$
Hence we have
$$
 \frac{k_n}{2n}+o\left( \frac{k_n}{n}\right) \leq \Re(\Phi_{w_m}(z_m, x_m))\leq 1-  \frac{k_n}{2n} + o\left( \frac{k_n}{n}\right)
 $$
 and $\Im(\Phi_{w_m}(z_m, x_m))= o(1)$.
 Since $r_{w_m}= |w_m|^{(1-\alpha)/2}\sim k_n/n$, we have \linebreak $\Phi_{w_m}(z_m, x_m)\in\mathcal{R}_{w_m}$ for $n$ large enough.
 It follows that $(z_m, x_m)\in\Phi^{-1}_{w_m}(\mathcal{R}_{w_m})\cap\Omega_{w_m}$ and hence by Property 3' we get
$$
(z_{m+1},x_{m+1})=\widetilde{F}_{w_m}(z_m, x_m)\in (V_{w_{m+1}}\times\mathbb{C})\cap\Omega_{w_{m+1}}
$$
and
$$\Phi_{w_{m+1}}(z_{m+1},x_{m+1}) = \Phi_{w_m}(z_m, x_m) + \frac{\sqrt{w_m}}{2}+o(w_m).$$
We obtain
\begin{align*}
& \Phi_{w_{m+1}}(z_{m+1}, x_{m+1}) \\
=&  \Phi_{v^\iota_n}(z^\iota_n, x^\iota_n)+\sum_{j=n^2+k_n}^{m-1}\left(\frac{\sqrt{w_j}}{2}+o\left(\frac{1}{n^2}\right)\right) +  \frac{\sqrt{w_m}}{2}+o\left(\frac{1}{n^2}\right)\\
=& \Phi_{v^\iota_n}(z^\iota_n, x^\iota_n)+\sum_{j=n^2+k_n}^{m}\left(\frac{\sqrt{w_j}}{2}+o\left(\frac{1}{n^2}\right)\right),
\end{align*}
which finishes the induction step.
We now take $m=(n+1)^2-k_n$ and set
$$v^o_n:=w_{(n+1)^2-k_n}, \qquad \text{and}\qquad (z^o_n, x^o_n):=\widetilde{\boldsymbol{F}}_{(n+1)^2-k_n, n^2+k_n}(z^\iota_n, x^\iota_n).$$
According to \eqref{Propos3Step21}, we have $(z^o_n, x^o_n)\in (V_{v^o_n}\times\mathbb{C})\cap \Omega_{v^o_n}$ and by \eqref{Propos3Step22},
\begin{align}\label{Propos3Step31}
\Phi_{v^o_n}(z^o_n, x^o_n) = \Phi_{v^\iota_n}(z^\iota_n, x^\iota_n) + \sum_{j=n^2+k_n}^{n^2+2n-k_n}\left(\frac{\sqrt{w_j}}{2}+o\left(\frac{1}{n^2}\right)\right).
\end{align}
Using \eqref{Propos3Step23} and Proposition \ref{propos2} we have
\begin{align*}
\Phi_{v^o_n}(z^o_n, x^o_n) = \frac{k_n}{2n}+o\left(\frac{k_n}{n}\right) + 1-\frac{k_n}{n}+o\left(\frac{1}{n}\right) = 1 - \frac{k_n}{2n}+o\left(\frac{k_n}{n}\right).
\end{align*}
Recall that $\Phi_w(z,x)=\varphi_w(z)$ does not depend on $x$, so we also have
\begin{align}\label{Propos3Step32}
\varphi_{v^o_n}(z^o_n)= 1 - \frac{k_n}{2n}+o\left(\frac{k_n}{n}\right),
\end{align}
We set
$$
X_n:= \frac{2}{\sqrt{v^o_n}}(\Phi_{v^o_n}(z^o_n, x^o_n) -1) = \frac{2}{\sqrt{v^o_n}}(\varphi_{v^o_n}(z^o_n) -1)
$$
so that
$$\varphi_{v^o_n}(z^o_n) = 1+ \frac{\sqrt{v^o_n}}{2} X_n.$$
Since $2/\sqrt{v^o_n}=2n+O(1)$, we infer from \eqref{Propos3Step32} that $X_n = -k_n(1+o(1))$. On the other hand $k_n\sim |v^o_n|^{-\alpha/2}$, so for $n$ large enough, $X_n\in D^{rep}_{v^o_n}$.
We saw in the proof of Property 2' that on the set $1+\frac{\sqrt{v^o_n}}{2} D^{rep}_{v^o_n}$ the map $\Phi_{v^o_n}^{-1}$ can be defined to map into $P^{rep}_{\widetilde{F}},$ so that we have
$$ \Phi_{v^o_n}^{-1}\left(1+\frac{\sqrt{v^o_n}}{2} X_n\right) = (z^o_n, u^o_n)$$
where $u^o_n\in\mathbb{C}$ is such that $(z^o_n, u^o_n)\in P^{rep}_{\widetilde{F}}$. Moreover Property 2' gives
$$\Psi_{\widetilde{F}}^{-1} \circ \Phi_{v^o_n}^{-1}\left(1+\frac{\sqrt{v^o_n}}{2} X_n\right) = X_n + o(1).$$
It follows that
$$\Psi_{\widetilde{F}}^{-1}(z^o_n, u^o_n) =X_n + o(1).$$
Using the definition of $X_n$, and \eqref{Propos3Step31}, we deduce
$$
\Psi_{\widetilde{F}}^{-1}(z^o_n, u^o_n) = \frac{2}{\sqrt{v^o_n}}\left(\Phi_{v^\iota_n}(z^\iota_n, x^\iota_n) + \sum_{j=n^2+k_n}^{n^2+2n-k_n}\left(\frac{\sqrt{w_j}}{2}\right)+o\left(\frac{1}{n}\right) -1 \right).
$$
Since $(z^\iota_n,x^\iota_n)\in\Phi_{\widetilde{F}}^{-1}(D^{att}_{v^\iota_n})\cap \Omega_{v^\iota_n},$ as shown in the beginning of this proof, we can use Property 1' to obtain
$$
\Psi_{\widetilde{F}}^{-1}(z^o_n, u^o_n) = \frac{2}{\sqrt{v^o_n}}\left( \frac{\sqrt{v^\iota_n}}{2}(\Phi_{\widetilde{F}}(z^\iota_n, x^\iota_n)+o(1)) + \sum_{j=n^2+k_n}^{n^2+2n-k_n}\left(\frac{\sqrt{w_j}}{2}\right)+o\left(\frac{1}{n}\right) -1 \right).
$$
Finally, we use $\sqrt{v^\iota_n}/2=1/(2n) + O(1/n^2)$ and $2/\sqrt{v^o_n}=2n+O(1)$ plus the fact that by Proposition \ref{propos2} the above summation is equal to $1-\frac{k_n}{n}+o(1/n)$ to get the desired result, namely
$$
\Psi_{\widetilde{F}}^{-1}(z^o_n, u^o_n) = \Phi_{\widetilde{F}} (z^\iota_n, x^\iota_n) - 2k_n +o(1).
$$
In particular we get $z^o_n \sim -1/\Psi^{-1}_{\widetilde{F}}(z^o_n, u^o_n) \sim -1/(k_n-2k_n) = 1/k_n.$
\hfill $\square$

\subsection{Proof of Proposition \ref{propos4}}

Let $z^o_n, x^o_n$ and $u^o_n$ be defined as before.
By Propositions \ref{propos1} and \ref{propos3}, we have
$$
\Psi^{-1}_{\widetilde{F}}(z^o_n, u^o_n)=-k_n+O(1).
$$
Hence we can choose $\kappa_1$ large enough so that
$$
\kappa_1>\Re(\Psi^{-1}_{\widetilde{F}}(z^o_n, u^o_n))+k_n+R+1.
$$
For $k\in[0,k_n-\kappa_1]$ define
$$(z^o_{n,k}, x^o_{n,k}):= \widetilde{\boldsymbol{F}}_{(n+1)^2-k_n+k, (n+1)^2-k_n}(z^o_n, x^o_n).$$
We also set $w_{n,k}:=w_{(n+1)^2-k_n+k}$.
In this proof we will assume that $P^{rep}_{\widetilde{F}}$ has been defined small enough so that $\widetilde{F}(P^{rep}_{\widetilde{F}})\subset \Psi_{\widetilde{F}}(\mathbb{C})$ is still a graph over $z$ with derivative bounded by $1$ in absolute value (the latter due to the fact that $P^{rep}_{\widetilde{F}}$ is tangent to the $z$-plane at the origin).
We want to show that
$$(z^o_{n,k_n-\kappa_1}, x^o_{n,k_n-\kappa_1}) = \widetilde{F}^{k_n-\kappa_1}(z^o_n, u^o_n) +o(1).$$

To do this, we will show by induction on $k$ that there exist constants $C, D>0$, not depending on $n$ or $k$, such that for $k\in[0,k_n-\kappa_1]$ the following assertions hold:
\begin{enumerate}[I(k)]
\item $u^o_{n,k}$ is well-defined by the condition $(z^o_{n, k}, u^o_{n,k})\in \widetilde{F}(P^{rep}_{\widetilde{F}})$,
\item $|u^o_{n,k}-x^o_{n,k}| \leq \frac{\left(\frac{1}{2}\right)^k C}{k_n^2} + \sum_{j=1}^{k} \frac{\left(\frac{1}{2}\right)^{j-1} D}{n^2},$
\item $|\Psi^{-1}_{\widetilde{F}}(z^o_{n, k}, u^o_{n,k}) - \Psi^{-1}_{\widetilde{F}}(z^o_{n}, u^o_{n}) - k | \leq 3D \frac{k k_n^2}{n^2} + \sum_{j=1}^{k} \frac{\left(\frac{1}{2}\right)^{j-1}C}{(k_n-j+1)^3} ,$ and
\item $(z^o_{n, k}, u^o_{n,k})\in P^{rep}_{\widetilde{F}}$.
\end{enumerate}

Let us show that Proposition \ref{propos4} follows from these hypotheses.

Since $D \sum_{j=0}^\infty \left(\frac{1}{2}\right)^j = 2D$ it follows from assertion II(k) that
\begin{align}\label{uxdistance}
|u^o_{n,k} - x^o_{n,k}|\leq \frac{\left(\frac{1}{2}\right)^k C}{k_n^2} +  \frac{2D}{n^2}  =o(1).
\end{align}
We also note that the right-hand side in assertion III(k) converges to $0$ for $k\in[0,k_n-\kappa_1]$ as $n\rightarrow \infty$. Indeed,
\begin{align*}
3D\frac{k k_n^2}{n^2} + \sum_{j=1}^{k} \frac{\left(\frac{1}{2}\right)^{j-1}C}{(k_n-j+1)^3} &\leq 3D \frac{ k_n^3}{n^2} +\sum_{j=1}^{k_n/2} \frac{\left(\frac{1}{2}\right)^{j-1}C}{(k_n-j+1)^3} + \sum_{j=k_n/2+1}^{k_n-\kappa_1} \frac{\left(\frac{1}{2}\right)^{j-1}C}{(k_n-j+1)^3} \\
&\leq 3D \frac{ k_n^3}{n^2} +\sum_{j=1}^{k_n/2} \frac{8C}{k_n^3}+\sum_{j=k_n/2+1}^{k_n}C \left(\frac{1}{2}\right)^{k_n/2}\\
&\leq  3D \frac{ k_n^3}{n^2} + \frac{4C}{k_n^2} + \frac{C \left(\frac{1}{2}\right)^{k_n/2}k_n}{2}\\
&=o(1),
\end{align*}
as $n\rightarrow\infty$, where we note that $k_n^3/n^2=o(1)$, because $\alpha<2/3$.
Hence, we obtain for the special case $k=k_n-\kappa_1$ that
$$
|\Psi^{-1}_{\widetilde{F}}(z^o_{n, k_n-\kappa_1}, u^o_{n,k_n-\kappa_1}) - \Psi^{-1}_{\widetilde{F}}(z^o_{n}, u^o_{n}) - k_n + \kappa_1 |= o(1).
$$
In other words,
$$
\Psi^{-1}_{\widetilde{F}}(z^o_{n, k_n-\kappa_1}, u^o_{n,k_n-\kappa_1}) = \Psi^{-1}_{\widetilde{F}}(z^o_{n}, u^o_{n}) + k_n - \kappa_1 +o(1) = \Psi^{-1}_{\widetilde{F}}(\widetilde{F}^{k_n-\kappa_1}(z^o_n, u^o_n)) + o(1).
$$
Applying $\Psi_{\widetilde{F}}$ on both sides yields
$$
(z^o_{n, k_n-\kappa_1}, u^o_{n,k_n-\kappa_1}) = \widetilde{F}^{k_n-\kappa_1}(z^o_n, u^o_n) + o(1).
$$
Since, in addition by \eqref{uxdistance} $u^o_{n,k_n-\kappa_1}=x^o_{n,k_n-\kappa_1}+o(1)$, this completes the proof of the main statement of Proposition \ref{propos4}.
Moreover, by IV($k_n-\kappa_1$) we have that $(z^o_{n, k_n-\kappa_1}, u^o_{n,k_n-\kappa_1})$ lies in  $P^{rep}_{\widetilde{F}}$ and in particular, $\widetilde{F}^{k_n-\kappa_1}(z^o_n, u^o_n)) $ lies in the set where the reverse coordinate change $U^{-1}$ is defined.

To begin with the induction, recall that for $k=0$ we have $z^o_{n,k}=z^o_n$ and $u^o_{n,k}=u^o_n$ is well-defined via $(z^o_n, u^o_n)\in P^{rep}_{\widetilde{F}}$ by Proposition \ref{propos3}.
Moreover, since \linebreak $(z^o_n, x^o_n)\in\Omega_{w_{n^2+2n-k_n}}$ and $z^o_n\sim 1/k_n$, we have $|x^o_n|<C|z^o_n|^2/2\sim C/(2k_n^2)$ for some constant $C>0$. Also, since $P^{rep}_{\widetilde{F}}$ is a holomorphic graph over the one-dimensional repelling petal that is tangent to the $z$-plane at the origin, by increasing $C$ if necessary, we have $|u^o_n|<C|z^o_n|^2/2\sim C/(2k_n^2)$ as well and $|u^o_n - x^o_n| < C/k_n^2$ follows, so the induction hypotheses hold for $k=0$.

Now suppose, that I(k), II(k), III(k) and IV(k) hold for some $k\in[0,k_n-\kappa_1-1]$.

\subsubsection*{Proof of I(k+1)}

We want to show that $u^o_{n,k+1}$ is well-defined. For that, note that
\begin{equation}\label{zkdist}
\begin{aligned}
|z^o_{n,k+1}  - \pi_z\circ\widetilde{F}(z^o_{n,k}, u^o_{n,k})| &= |\pi_z\circ \widetilde{F}_{w_{n,k}}(z^o_{n,k}, x^o_{n,k})- \pi_z\circ \widetilde{F}(z^o_{n,k}, u^o_{n,k})|\\
&\leq  |\pi_z\circ \widetilde{F}(z^o_{n,k}, x^o_{n,k})- \pi_z\circ \widetilde{F}(z^o_{n,k}, u^o_{n,k})+ O(w_{n,k})|\\
&\leq \sup \left| \frac{\partial (\pi_z\circ \widetilde{F})}{\partial x}\right| |u^o_{n,k}-x^o_{n,k}| +\frac{D}{n^2}\\
&\leq  \frac{1}{(k_n-k)^3} | u^o_{n,k} - x^o_{n,k}|+\frac{D}{n^2}\\
&= o(1),
\end{aligned}
\end{equation}
as $n\rightarrow\infty$. Here, and also in equation \eqref{xdistance} below, the supremum is taken over the vertical interval between the points $(z^o_{n,k}, u^o_{n,k})$ and $(z^o_{n,k}, x^o_{n,k})$. We use that $w_{n,k}$ is of order $O(1/n^2)$ by \eqref{wn}. The second to last step in \eqref{zkdist} follows from $\partial( \pi_z\circ \widetilde{F})/\partial x = O(z^4)$ and $z^o_{n,k} \sim -1/\Psi^{-1}_{\widetilde{F}}(z^o_{n,k}, u^o_{n,k}) \sim -1/(k-k_n)$ by the induction asumption III(k). The last step follows from \eqref{uxdistance} .

Since $(z^o_{n, k}, u^o_{n,k})\in P^{rep}_{\widetilde{F}}$ by IV(k), we have $\widetilde{F}(z^o_{n, k}, u^o_{n,k})\in \widetilde{F}(P^{rep}_{\widetilde{F}})$ and since $z^o_{n,k+1}$ is close to $\pi_z\circ\widetilde{F}(z^o_{n,k}, u^o_{n,k})$ by \eqref{zkdist}, we can define $u^o_{n,k+1}$ by requiring that $(z^o_{n,k+1}, u^o_{n,k+1})$ lies in $\widetilde{F}(P^{rep}_{\widetilde{F}})\subset \Psi_{\widetilde{F}}(\mathbb{C})$.
Therefore $u^o_{n,k+1}$ is well-defined.

\subsubsection*{Proof of II(k+1)}

Observe that
$$
|u^o_{n, k+1} - x^o_{n,k+1}| \leq |u^o_{n,k+1} - \pi_x\circ \widetilde{F}(z^o_{n,k}, u^o_{n,k})| + | \pi_x\circ \widetilde{F}(z^o_{n,k}, u^o_{n,k}) -  \pi_x\circ \widetilde{F}_{w_{n,k}}(z^o_{n,k}, x^o_{n,k})|.
$$
We estimate the first absolute difference by noting that $u^o_{n,k+1}$ and $\pi_x\circ \widetilde{F}(z^o_{n,k}, u^o_{n,k})$ are the $x$-coordinates of the graph corresponding to $\widetilde{F}(P^{rep}_{\widetilde{F}})$ evaluated in $z^o_{n,k+1}$ and in $\pi_z\circ \widetilde{F}(z^o_{n,k}, u^o_{n,k})$. Since this graph has slope bounded by $1$, we get
\begin{equation}\label{uxnext}
\begin{aligned}
 |u^o_{n,k+1} - \pi_x\circ \widetilde{F}(z^o_{n,k}, u^o_{n,k})| &< |z^o_{n,k+1} - \pi_z\circ \widetilde{F}(z^o_{n,k}, u^o_{n,k})|\\
 &< \frac{|u^o_{n,k}-x^o_{n,k}|}{4}+  \frac{D}{n^2} ,
\end{aligned}
\end{equation}
as estimated in \eqref{zkdist} and using that $(k_n-k)^3>4$.
On the other hand, since the $x$-coordinates of $F_w$ and $F$ are the same,
\begin{equation}\label{xdistance}
\begin{aligned}
 | \pi_x\circ \widetilde{F}(z^o_{n,k}, u^o_{n,k}) -  \pi_x\circ \widetilde{F}_{w_{n,k}}(z^o_{n,k}, x^o_{n,k})| &= | \pi_x\circ \widetilde{F}(z^o_{n,k}, u^o_{n,k}) -  \pi_x\circ \widetilde{F}(z^o_{n,k}, x^o_{n,k})|\\
  &<\sup|\frac{\partial \pi_x\circ \widetilde{F}}{\partial x}| |u^o_{n,k}-x^o_{n,k}|.
\end{aligned}
\end{equation}
Since $|\frac{\partial \pi_x\circ \widetilde{F}}{\partial x}| = \delta + O(x^2,z) <1/4,$ if $\delta$ small enough, we can combine \eqref{uxnext} and \eqref{xdistance} to get
\begin{align*}
|u^o_{n,k+1} - x^o_{n,k+1}| &\leq \frac{1}{2}|u^o_{n,k}-x^o_{n,k}| +  \frac{D}{n^2} \\
&\leq \frac{1}{2}\left(\frac{\left(\frac{1}{2}\right)^kC}{k_n^2}+ \sum_{j=1}^{k}\frac{\left(\frac{1}{2}\right)^{j-1} D}{n^2}\right) +  \frac{D}{n^2} \\
&=\frac{ \left(\frac{1}{2}\right)^{k+1}C}{k_n^2} +  \sum_{j=1}^{k+1} \frac{\left(\frac{1}{2}\right)^{j-1} D}{n^2},
\end{align*}
which proves assertion II(k+1).

\subsubsection*{Proof of III(k+1)}

We recall from \eqref{zkdist} that
\begin{align}\label{zkdistance}
|z^o_{n,k+1} - \pi_z \circ \widetilde{F}(z^o_{n,k},u^o_{n,k})| \leq  \frac{1}{(k_n-k)^3} | u^o_{n,k} - x^o_{n,k}| +\frac{D}{n^2}.
\end{align}
Combining this with \eqref{uxdistance}, we obtain
\begin{align*}
|z^o_{n,k+1} - \pi_z \circ \widetilde{F}(z^o_{n,k},u^o_{n,k})| &\leq   \frac{1}{(k_n-k)^3}  \left( \frac{2D}{n^2}+\frac{\left(\frac{1}{2}\right)^k C}{k_n^2}\right)+  \frac{D}{n^2}\\
&\leq \frac{3D}{n^2} + \frac{\left(\frac{1}{2}\right)^k C}{k_n^2 (k_n-k)^3}.
\end{align*}
Recall that $\Psi_{\widetilde{F}}^{-1}(z,x)\sim -1/z$ is defined on a graph, hence can be considered as a function of $z$ whose derivative is of order $O(1/z^2)$.
Using $|z^o_{n,k}|\sim|1/(k-k_n)|\geq1/k_n$ we infer that
\begin{align*}
|\Psi^{-1}_{\widetilde{F}}(z^o_{n, k+1}, u^o_{n,k+1}) -  \Psi^{-1}_{\widetilde{F}} \circ \widetilde{F}(z^o_{n,k}, u^o_{n,k})| \leq k_n^2 \left( \frac{3D}{n^2} + \frac{\left(\frac{1}{2}\right)^k C}{k_n^2 (k_n-k)^3}\right).
\end{align*}
Combining this error with induction hypothesis III(k) gives
\begin{align*}
&|\Psi^{-1}_{\widetilde{F}}(z^o_{n, k+1}, u^o_{n,k+1}) - \Psi^{-1}_{\widetilde{F}}(z^o_{n}, u^o_{n}) - (k+1) | \\
\leq &|\Psi^{-1}_{\widetilde{F}}(z^o_{n, k+1}, u^o_{n,k+1}) - \Psi^{-1}_{\widetilde{F}}(\widetilde{F} (z^o_{n, k}, u^o_{n,k})) | +  |\Psi^{-1}_{\widetilde{F}}(z^o_{n, k}, u^o_{n,k}) - \Psi^{-1}_{\widetilde{F}}(z^o_{n}, u^o_{n}) - k | \\
\leq &k_n^2 \left( \frac{3D}{n^2} + \frac{\left(\frac{1}{2}\right)^k}{k_n^2 (k_n-k)^3}\right)  + 3D\frac{k k_n^2}{n^2} + \sum_{j=1}^{k} \frac{\left(\frac{1}{2}\right)^{j-1} C}{(k_n-j+1)^3} \\
\leq& 3D\frac{(k+1) k_n^2}{n^2} + \sum_{j=1}^{k+1} \frac{\left(\frac{1}{2}\right)^{j-1}C }{(k_n-j+1)^3},
\end{align*}
which finishes the proof of III(k+1).

\subsubsection*{Proof of IV(k+1)}

By III(k+1) we have
$$
|\Psi^{-1}_{\widetilde{F}}(z^o_{n, k+1}, u^o_{n,k+1}) - \Psi^{-1}_{\widetilde{F}}(z^o_{n}, u^o_{n}) - (k+1) |= o(1),
$$
as $n\rightarrow\infty$.
From this we see that
\begin{align*}
\Re(\Psi^{-1}_{\widetilde{F}}(z^o_{n, k+1}, u^o_{n,k+1}) ) &= \Re(\Psi^{-1}_{\widetilde{F}}(z^o_{n}, u^o_{n}))+ k +1+ o(1) \\
&< \kappa_1-k_n - R -1 +k + 1 +o(1)\leq -R - 1 + o(1) <-R.
\end{align*}
In particular we obtain that
$(z^o_{n,k+1}, u^o_{n,k+1})\in P^{rep}_{\widetilde{F}}$, which completes the induction and thereby the proof of Proposition \ref{propos4}.
\hfill $\square$

  \bibliographystyle{abbrv}
  \bibliography{mybib}

\end{document}